\documentclass[10pt，oneside, reqno]{amsart}

\usepackage{epsfig}
\usepackage{amsthm}
\usepackage{amssymb}
\usepackage{amsmath}
\usepackage{amscd}
\usepackage{color}
\usepackage{esint}
\usepackage{enumitem}
\usepackage{bbm}
\usepackage{mathrsfs}
\usepackage[top=1.1in, bottom=1in, left=1in, right=1in]{geometry}

\usepackage[usenames,dvipsnames]{xcolor}
\usepackage[colorlinks=true, pdfstartview=FitV, linkcolor=blue, citecolor=blue, urlcolor=blue]{hyperref}

\newcommand{\lb}{\varLambda}

\def \<{\langle}
\def \>{\rangle}

\newcommand{\bg}{\begin{equation}}
\newcommand{\ed}{\end{equation}}
\newcommand{\bga}{\begin{eqnarray}}
\newcommand{\eda}{\end{eqnarray}}

\def\cbdu{\par{\raggedleft$\Box$\par}}

\newtheorem {Theorem}  {Theorem}
\newtheorem {Corollary}[Theorem]{\bf Corollary}

\numberwithin{Theorem}{section}

\newtheorem {Lemma}[Theorem]  {Lemma}
\newtheorem {Proposition}[Theorem]{Proposition}
\theoremstyle{definition}
\newtheorem{Definition}[Theorem]{Definition}
\theoremstyle{remark}
\newtheorem{Remark}[Theorem]{\bf Remark}

\def \l{\lambda}
\expandafter\chardef\csname pre amssym.def
at\endcsname=\the\catcode`\@ \catcode`\@=11
\def\undefine#1{\let#1\undefined}
\def\newsymbol#1#2#3#4#5{\let\next@\relax
 \ifnum#2=\@ne\let\next@\msafam@\else
 \ifnum#2=\tw@\let\next@\msbfam@\fi\fi
 \mathchardef#1="#3\next@#4#5}
\def\mathhexbox@#1#2#3{\relax
 \ifmmode\mathpalette{}{\m@th\mathchar"#1#2#3}%
 \else\leavevmode\hbox{$\m@th\mathchar"#1#2#3$}\fi}
\def\hexnumber@#1{\ifcase#1 0\or 1\or 2\or 3\or 4\or 5\or 6\or 7\or 8\or
 9\or A\or B\or C\or D\or E\or F\fi}

\font\teneufm=eufm10 \font\seveneufm=eufm7 \font\fiveeufm=eufm5
\newfam\eufmfam
\textfont\eufmfam=\teneufm \scriptfont\eufmfam=\seveneufm
\scriptscriptfont\eufmfam=\fiveeufm

\catcode`\@=\csname pre amssym.def at\endcsname

\newcounter{remark}
\numberwithin{equation}{section}
\numberwithin{figure}{section}

\def \grad {\nabla}

\newcommand{\divv}{{\text {div}}\,}
\newcommand{\curl}{\nabla \times}

\newcommand{\e}{\epsilon}
\renewcommand{\d}{\delta}
\newcommand{\om}{\omega}

\renewcommand{\k}{\kappa}
\renewcommand{\l}{\lambda}
\newcommand{\Om}{\Omega}
\renewcommand{\a}{\alpha}

\renewcommand{\b}{\beta}

\newcommand{\gm}{\gamma}
\newcommand{\s}{\sigma}
\newcommand{\Gm}{\Gamma}
\newcommand{\R}{\mathbf{R}}

\newcommand{\les}{\lesssim}

\newcommand{\Ff}{{\mathcal F}}

\newcommand{\Gg}{{\mathcal G}}

\newcommand{\Pp}{{\mathcal P}}

\newcommand{\Ss}{{\mathcal S}}
\newcommand{\Tt}{{\mathcal T}}

\newcommand{\barB}{\overline{B}}

\newcommand{\tFf}{\widetilde{\Ff}}
\newcommand{\tF}{\widetilde{\F}}
\newcommand{\tP}{\widetilde{\P}}

\newcommand{\tB}{\widetilde{B}}
\newcommand{\tW}{\widetilde{\W}}
\newcommand{\tOm}{\widetilde{\Om}}

\newcommand{\tSs}{\widetilde{\Ss}}

\def  \R   {{\mathbb R}}
\def  \Z   {{\mathbb Z}}
\def  \P   {{\mathbb P}}
\def  \E    {{\mathbb E}}
\def  \N   {{\mathbb N}}
\def  \F   {{\mathbb F}}

\def  \T   {{\mathbb T}}
\def  \W   {{\mathbb W}}

\def  \haf  {{\frac{1}{2}}}
\def  \p   {\partial}

\def  \tr   {\operatorname{Tr}}

\def  \Dq    {\Delta_q}

\def  \sumk  {\sum_{k=1}^\infty}
\def  \sumj  {\sum_{j=1}^\infty}

\def  \chir   {\chi_r}
\def  \tphi   {\widetilde{\phi}}
\def  \tJ     {\widetilde{J}}
\def  \tw     {\widetilde{W}}

\def  \wtDq   {\widetilde{\Delta}_q}
\def  \wtDl   {\widetilde{\Delta}_l}
\def  \Dl     {\Delta_l}
\def  \smh    {{s-\frac{1}{2}}}

\newcommand\twonorm[1]{\lVert#1\rVert_{L^2(\T^3)}}

\newcommand\rnorm[1]{\Vert#1\Vert_{L^r(\T^3)}}
\newcommand\sixnorm[1]{\Vert#1\Vert_{L^6(\T^3)}}
\newcommand\Linfnorm[1]{\Vert#1\Vert_{L^\infty(\T^3)}}

\newcommand\Honenorm[1]{\Vert#1\Vert_{H^1(\T^3)}}
\newcommand\Hsnorm[1]{\Vert#1\Vert_{H^s(\T^3)}}

\newcommand\Woneinfnorm[1]{\left \lVert #1 \right \rVert_{W^{1,\infty}}}

\newcommand\LinfNorm[1]{\left \lVert#1\right \rVert_{L^\infty(\T^3)}}

\newcommand\HsNorm[1]{\left\lVert#1\right \rVert_{H^s(\T^3)}}
\newcommand\HssNorm[1]{\left\lVert#1\right \rVert_{H^{s+1}(\T^3)}}
\newcommand\HsaNorm[1]{\left\lVert#1\right \rVert_{H^{s+\frac{\a}{2}}(\T^3)}}
\newcommand\HsmhNorm[1]{\left\lVert#1\right \rVert_{H^{s-\haf}(\T^3)}}

\newcommand{\Ltwoinner}[1]{\left < #1 \right >}
\newcommand{\ltwoHssnorm}[1]{\left \lVert  #1 \right \rVert_{\ell(\N,H^{s+1}(\T^3))}}

\def\build#1_#2^#3{\mathrel{\mathop{\kern 0pt#1}\limits_{#2}^{#3}}}

\begin{document}

\title[Three-Dimensional Stochastic EMHD System]{The Three-Dimensional Stochastic EMHD System: Local Well-Posedness and Maximal Pathwise Solutions}


\author[Ruimeng Hu]{Ruimeng Hu\textsuperscript{1}}\thanks{\textsuperscript{1}Department of Mathematics, Department of Statistics and Applied Probability, University of California, Santa Barbara, Santa Barbara, CA 93106-3080, USA. Email: 
rhu@ucsb.edu}

\author[Qirui Peng]{Qirui Peng\textsuperscript{2}}
\thanks{\textsuperscript{2}Department of Mathematics, University of California, Santa Barbara, Santa Barbara, CA 93106-3080, USA.
Emails: \{qpeng9, xy6\}@ucsb.edu}

\author[Xu Yang]{Xu Yang\textsuperscript{2}}

\thanks{}





\begin{abstract} 
We study the three-dimensional stochastic electron-magnetohydrodynamics (EMHD) system with fractional dissipation on the torus, driven by Stratonovich transport noise acting through divergence-free first-order operators and producing an It\^o correction while preserving the transport-type structure of the Hall nonlinearity. The Hall term is derivative intensive, and in the stochastic setting it needs to be controlled together with commutators generated by the transport operators. We develop a high-order Sobolev energy method based on Littlewood--Paley analysis and refined commutator estimates, yielding uniform bounds for Galerkin approximations in $H^s$ with $s>\tfrac{5}{2}$ together with suitable time regularity. Using stochastic compactness and identification of limits, we construct martingale solutions for initial data in $L^2(\Omega;H^s)$. Pathwise uniqueness is established through cancellations in the Hall term and a stochastic Gr\"onwall argument, and the Yamada--Watanabe--type theorem then yields local pathwise well-posedness and maximal pathwise solutions.

\bigskip

\noindent \textbf{Keywords.} Electron-magnetohydrodynamics, Stratonovich transport noise, well-posedness, maximal pathwise solutions.

\medskip

\noindent \textbf{MSC Classification.} 35R60, 35Q35, 60H15, 76M35.
\end{abstract}

\maketitle

\section{Introduction}\label{Sec_Intro}  
Electron-magnetohydrodynamics arises as a fundamental model in plasma physics for describing the interaction between magnetic fields and charged fluid motion. A broader framework is provided by the generalized Hall-magnetohydrodynamics (Hall-MHD) equations:
\begin{subequations}
\begin{align}
\p_t u + (u \cdot \grad) u + \grad p &= \Delta u + (B \cdot \grad) B, \label{eq:Hall-MHD_1} \\
\p_t B + (u \cdot \grad) B + \curl \bigl( (\curl B) \times B \bigr) + \lb^{\a}B &= B \cdot \grad u, \label{eq:Hall-MHD_2} \\
\nabla \cdot u = \nabla \cdot B &= 0, \label{eq:Hall-MHD_3} \\
u(0) &= u_0, \label{eq:Hall-MHD_4}\\
B(0) &= B_0, \label{eq:Hall-MHD_5}
\end{align}
\end{subequations}
where we consider the solutions $(u,B)$ on the domain $[0,\infty) \times \T^3$. Here $u = u(x,t)$ and $p = p(x,t)$ represent the velocity fields of the fluid and scalar pressure, respectively.  $B = B(x,t)$ denotes the magnetic field. The parameter $\a > 0$ determines the strength of magnetic resistance and the generalized Laplacian $\lb^\a : = (-\Delta)^\frac{\a}{2}$ is defined as a Fourier multiplier:
\begin{equation}\label{def:Lambda}
\widehat{\lb^\a B} (k) = |k|^{\a} \widehat{B}. 
\end{equation}
A central analytical difficulty in establishing the well-posedness of the Hall-MHD system arises from the highly nonlinear Hall term $\curl (\curl B) \times B$. A simplified model of the system $\eqref{eq:Hall-MHD_1} - \eqref{eq:Hall-MHD_5}$ that still captures the nonlinear Hall term is known as the Electron-MHD (EMHD):
\begin{align}
\p_t B + \curl \bigl( (\curl B) \times B \bigr) + \lb^{\a}B &= 0, \label{eq:EMHD_1} \\
\grad \cdot B &= 0. \label{eq:EMHD_2}
\end{align}
In this paper, we study the stochastic EMHD system in the presence of Stratonovich transport noise: 
\begin{subequations}\label{eq:EMHD_Stratonovich}
\begin{align}
dB + \bigl ( \curl ( (\curl B) \times B ) + \lb^{\a}B \bigr ) dt &= \sumk \left( c_k \cdot \grad B \right) \circ dW^k, \label{eq:EMHD_Stratonovich_1} \\
\grad \cdot B &= 0, \label{eq:EMHD_Stratonovich_2} \\
B(0) &= B_0, \label{eq:EMHD_Stratonovich_3} 
\end{align}
\end{subequations}
where the sequence $\{W^k\}_{k \ge 1}$ denotes a family of independent standard Brownian motions, while $\{c_k\}_{k \ge 1}$ are divergence-free vector fields corresponding to the coefficients of the transport noise. 

It admits the equivalent It\^o form (see \cite[Section 8]{AHHS25}),
\begin{subequations}\label{eq:EMHD_Ito}
\begin{align}
dB + \bigl ( \curl ( (\curl B) \times B ) + \lb^{\a}B \bigr ) dt &= \haf \sum_{k =1}^\infty  \Tt^2_k B dt + \sum_{k=1}^\infty \Tt_k B dW^k, \label{eq:EMHD_Ito_1} \\
\nabla \cdot B &= 0, \label{eq:EMHD_Ito_2} \\
B(0) &=  B_0, \label{eq:EMHD_Ito_3}
\end{align}
\end{subequations}
where the linear operator $\Tt_k$ is defined by
\[
\Tt_k B := (c_k \cdot \grad ) B.
\]
This It\^o formulation highlights the analytical structure of the problem: the transport noise generates both first-order stochastic terms and second-order correction operators, which need to be treated together with the derivative-intensive Hall nonlinearity. 

The main analytical difficulty of the present work lies precisely in this interplay between the Hall nonlinearity and the stochastic transport noise. Unlike classical stochastic fluid models, the Hall term $\curl\left((\curl B)\times B\right)$ contains one additional derivative, which leads to a substantial loss of regularity in high-order Sobolev estimates. At the same time, the stochastic terms generate additional commutator structures that have to be controlled at the same regularity level. This difficulty is further amplified in the regime $\alpha<2$, where the fractional resistance is not sufficiently strong to directly compensate for the derivative loss.

To overcome this obstacle, we develop a cutoff approximation framework together with refined high-order energy estimates based on Littlewood--Paley analysis and commutator estimates. More specifically, as a key component of our analytical framework, we introduce the following cutoff approximation scheme, i.e., we choose a positive non-increasing function $\chi_r \in C^\infty(\R)$ as
\begin{align}\label{def:cutoff}
\chi_r (x) = \begin{cases}
1,\ \ \text{if} \ \  |x| \leq \frac{r}{2},\\
0, \ \ \text{if}\ \  |x| > r.
\end{cases}
\end{align}
We then consider the cutoff of the system \eqref{eq:EMHD_Ito}: 
\begin{align}
dB + \bigl (\chi_r^2 \curl ( (\curl B) \times B ) + \lb^{\a}B \bigr ) dt &= \haf \sum_{k =1}^\infty  \Tt^2_k B dt + \sum_{k=1}^\infty \Tt_k B dW^k, \label{eq:EMHD_cutoff_Ito_1} \\
B(0) &=  B_0, \label{eq:EMHD_cutoff_Ito_2} \\
\nabla \cdot B &= 0, \label{eq:EMHD_cutoff_Ito_3}
\end{align}
where we write $\chi_r := \chi_r(\Woneinfnorm{B})$ for simplicity.

\smallskip

\noindent {\bf Main result.} The above approximation framework allows us to establish the following local well-posedness result in the pathwise sense.

\begin{Theorem}\label{Thm:main}
Suppose that $s > 3$, $\a \in (1,2]$ and the noise coefficients satisfy the regularity assumption \eqref{eq:noise_assumption}. Let $B_0 \in L^2(\Om, H^s)$ be any initial data. Then there exists a unique maximal pathwise solution $\bigl( B, (\eta_n)_{n \geq 1}, \xi \bigr)$ to the system \eqref{eq:EMHD_Stratonovich}.
\end{Theorem}

\noindent {\bf Related literature.}
On the deterministic side, the well-posedness theory for Hall--MHD and EMHD has been extensively studied over the past decade; we refer the reader to \cite{CDL14,CWW15,WZ17,Dai21} and the references therein. More recently, stochastic effects in EMHD-type models have begun to receive attention. In \cite{hu2025well}, we studied a three-dimensional EMHD model without resistivity, in which the resistive mechanism is replaced by multiplicative noise. In that setting, stochastic perturbations were shown to restore local well-posedness in suitable Gevrey spaces and global well-posedness with high probability for small initial data. In contrast, the stochastic theory for EMHD systems with transport noise and fractional dissipation remains largely unexplored. On the stochastic fluid side, transport noise was originally introduced by Kraichnan (see \cite{kraichnan1968small,kraichnan1994anomalous}) to model turbulent effects such as anomalous diffusion, and has since been investigated in a variety of stochastic fluid equations, including the stochastic Navier--Stokes equations \cite{mikulevicius2001equations,MR04,hu2023pathwise,abdo2023long} and the primitive equations \cite{agresti2023primitive,AHHS25,HLL26,hu2023local,hu2025regularization}. This paper is of a different nature from \cite{hu2025well}: rather than replacing the resistive term with multiplicative noise, we consider Stratonovich transport noise acting directly on the magnetic field dynamics and establish local pathwise well-posedness together with the existence of maximal pathwise solutions. This requires a new analytical framework that handles simultaneously the derivative-intensive Hall nonlinearity and the commutator structures generated by the transport noise.

The rest of the paper is organized as follows. Section \ref{Sec_Prelim} recalls the function spaces, Littlewood--Paley theory, probability space, and solution concepts used throughout the paper. In Section \ref{Sec:existence}, we prove the existence of martingale solutions by deriving uniform energy estimates for the Galerkin approximations and passing to the limit through a compactness argument. Section \ref{Sec:proof_of_main_Thm} establishes pathwise uniqueness and completes the proof of the main result. We conclude with a brief discussion in Section \ref{Sec:conclusion}. Finally, additional technical lemmas and background from stochastic calculus are collected in Appendices \ref{appendix:functional_analysis}--\ref{appendix:commutator_estimates}.

\section{Preliminary}\label{Sec_Prelim}
In this section, we recall the function spaces, Littlewood--Paley theory, the probability space, and the solution concepts that will be used throughout the paper.
\subsection{Function spaces} 
For a function $f$ with domain $\T^3:= \R^3/\Z^3$ and $p \in \left[1,\infty \right)$, define its $L^p$-norm by
\[
\| f \|_{L^p(\T^3)} = \biggl( \int_{\T^3} |f(x)|^p dx \biggr)^\frac{1}{p},
\]
and $L^\infty$-norm by 
\[
\Linfnorm{f} := \inf \left \{ \l > 0: \mu \bigl( \{ x: |f(x)| > \l\} \right) =0 \bigr \}.
\]

If $f,g \in L^2(\T^3)$, we denote their $L^2$-inner product by
\[
\left<f,g \right> = \int_{\T^3} f(x) g(x) dx.
\]

Furthermore, the Fourier expansion of $f \in L^2(\T^3)$ is given by
\[
f(x) = \sum_{k \in \Z^3} \widehat{f}_k e^{2\pi i k \cdot x}, \ \ \ \widehat{f}_k := \int_{\T^3} f(x) e^{-2\pi i k \cdot x} dx,
\]
and we call $\widehat{f}_k$ the $k^{\text{th}}$ Fourier coefficient. 

For $s \in \R$, we write $H^s(\T^3)$ for the Sobolev space on $\T^3$ with norm 
\[
\Hsnorm{f} := \biggl( \sum_{k \in \Z^3} \left(1 + |k|^{2s} \right) |\widehat{f}_k|^2 \biggr)^\haf.
\]

If $f$ has zero mean, then its $H^s$-norm is equivalent to the semi-norm $\dot{H}^s$ given by
\[
\|f \|_{\dot{H}^s(\T^3)} := \biggl( \sum_{k \in \Z^3} |k|^{2s} |\widehat{f}_k|^2 \biggr)^\haf. 
\]

Finally, we define $H$ as the space of $L^2$ function that are divergence-free,
\[
H:= \left\{ f \in L^2 : \divv f = 0 \right \}.
\]

\subsection{Littlewood--Paley theory}\label{LP_Intro}
We present a short overview of Littlewood–Paley theory on $\T^3$. For a more thorough exposition, the reader may consult the classical references by Bahouri, Chemin, and Danchin~\cite{bahouri2011fourier} and Grafakos~\cite{grafakos2008classical}. We begin by fixing a nonnegative radial function $\chi \in C_0^\infty(\R^3)$ satisfying
\begin{equation} \label{def:chi}
\chi(\xi) :=
\begin{cases}
1, & \text{for } |\xi| \le \tfrac{3}{4},\\
0, & \text{for } |\xi| \ge 1.
\end{cases}
\end{equation}
Next we define
\begin{equation}\label{def:varphi}
\varphi(\xi) := \chi(\xi/2) - \chi(\xi),  
\end{equation}
and let
\begin{equation*}
\varphi_q(\xi) :=
\begin{cases}
\varphi(\l_q^{-1}\xi), & q \ge 0,\\
\chi(\xi), & q = -1,
\end{cases}
\end{equation*}
where $\l_q = 2^q$. In Fourier space, this construction yields a dyadic partition of unity given by the family $\{\varphi_q\}_{q\ge -1}$. For a tempered distribution $u$ on $\T^3$, we define its $q$-th Littlewood–Paley projection by
\[
  \Delta_q u(x) := \sum_{k\in\Z^3} \hat{u}(k)\,\varphi_q(k)\,e^{ 2\pi ik \cdot x}.
\]
With this definition, one has
\[
u = \sum_{q=-1}^\infty \Delta_q u,
\]
in the sense of distribution. For every $q \in \mathbb{N}$, we introduce the cutoff operator
\[
\Ss_q u := \sum_{j=-1}^q \Delta_j u.
\]

\noindent The $H^s$-norm of $u$ can be equivalently characterized by the Littlewood--Paley decomposition as
\begin{equation}\label{def_Hs_norm}
    \|u\|_{H^s(\T^3)} := \biggl(\sum_{q=-1}^\infty \lambda_q^{2s}\|\Dq u\|_{L^2(\T^3)}^2\biggr)^{1/2},
\end{equation}
for any $u \in H^s(\T^3)$ and $s\in\mathbb{R}$. For notational convenience, we define
\begin{equation*}
\wtDq u := \Delta_{q-1} u+\Delta_{q} u+ \Delta_{q+1} u.
\end{equation*}

\noindent We recall the Bernstein inequalities satisfied by each dyadic block in the Littlewood--Paley decomposition.
\begin{Lemma}\label{lemma:Bernstein}(Bernstein's inequality) 
Let $d$ be the spatial dimension. Let $1 \le s \le r \le \infty$ and $k \ge 0$. Then for any tempered distribution $u$,
\bg\label{Bern1}
\lambda_q^{k}\|\Dq u\|_{L^r (\T^d)} \lesssim \|\nabla^k \Dq u\|_{L^r (\T^d)}\lesssim \lambda_q^{k+d(\frac{1}{s}-\frac{1}{r})}\|\Dq u\|_{L^s(\T^d)}.
\ed
\end{Lemma}
\noindent 
We will make use of the following paraproduct formula to estimate the Hall-term later on. 
\begin{Lemma}(Bony's paraproduct formula)\label{lemma:Bony}
Let $u$ and $v$ be tempered distributions. Then
\begin{equation*}
u\cdot v
=\sum_{l = -1}^\infty \big(\Ss_{l-2} u\cdot  \Delta_l v\big)
+\sum_{l = -1}^\infty \big(\Dl u \cdot \Ss_{l-2} v \big)
+\sum_{l=-1}^\infty \big(\Dl u \cdot \wtDl v\big).
\end{equation*}
\end{Lemma}
\noindent It then follows that, for example,
\begin{equation*}
\Delta_q(u\cdot\nabla v)
=\sum_{|q-l|\leq 2}\Delta_q\big(\Ss_{l-2} u\cdot\nabla \Delta_l v\big)
+\sum_{|q-l|\leq 2}\Delta_q\big(\Dl u\cdot\nabla \Ss_{l-2} v\big)
+\sum_{l\geq q-2} \Delta_q\big(\Dl u \cdot\nabla \wtDl v\big).
\end{equation*}

\subsection{Probability space} 
Let $\Ss = (\Om, \Ff, \F, \P)$ be a stochastic basis with filtration 
$\F = \{\Ff_t\}_{t \ge 0}$. Let $\mathscr{U}$ be a separable Hilbert space, and let 
$\W$ be an $\F$-adapted cylindrical Wiener process on $\mathscr{U}$ defined on $\Ss$. 
Denote by $\{e_k\}_{k \in \N}$ an orthonormal basis of $\mathscr{U}$. Then $\W$ admits 
the representation
\[
\W = \sum_{k \in \N} e_k W^k,
\]
where $\{W^k\}_{k \in \N}$ are independent standard Brownian motions on $\Ss$. Let $T: \mathscr{U} \to H$ be the linear operator defined by 
\[
T e_k = c_k, \ \ \ k \geq 1.
\]
Then the noise term in \eqref{eq:EMHD_Stratonovich} is obtained by 
\[
\sumk \left( c_k \cdot \grad B \right) \circ dW^k = \left( T \cdot \grad B\right) \circ d\W.
\]
We further require the noise coefficient vectors $\{ c_k\}_{k \geq 1}$ to satisfy 
\begin{equation}\label{eq:noise_assumption}
\{ c_k\}_{k \geq 1} \in \ell^2 (\N, H^{s+1}).
\end{equation}
\subsection{Pathwise and Martingale solutions}
In this section, we introduce the definitions of pathwise and martingale solutions. Although a pathwise solution is stronger than a martingale solution in the sense of stochastic PDE theory, both notions possess the same level of physical regularity as classical solutions.

\begin{Definition}[Martingale solution]\label{def_Martingale_solution} For $T > 0$, let $B_0 \in L^2(\Om; H^s)$ be an $\Ff_0$-measurable $H^s$-valued random variable. We call a quadruple $\bigl(\tB_0, \tB, \tW, \tSs \bigr)$ a martingale solution to the system \eqref{eq:EMHD_cutoff_Ito_1}-\eqref{eq:EMHD_cutoff_Ito_3} on $[0,T]$ if:
\begin{enumerate}
\item $\widetilde{\Ss} = \bigl(\tOm, \tFf, \tF, \tP \bigr)$ is a stochastic basis such that $\tW$ is an $\tF$-adapted cylindrical Brownian motion with components $\{\widetilde{W}^k\}_{k \ge 1}$, \\
\item $\tB_0 \in L^2(\tOm,H^s)$ has the same law as $B_0$, \\
\item $\tB$ is a progressively measurable process such that 
\[
\tB \in L^2\bigl(\tOm; C([0,T];H^s)\bigr), \ \ \ \tP\text{-a.s.},
\]
and for every $t\in [0,T]$ the following identity holds: 
\begin{align*}
&\Ltwoinner{\tB(t),\phi} + \int_0^t \Ltwoinner{\chi^2_r \curl \bigl( (\curl \tB ) \times \tB \bigr) + \lb^\a \tB, \phi} ds \\
=\ &\Ltwoinner{\tB(0),\phi} + \haf \sumk \int_0^t \Ltwoinner{\Tt^2_k \tB,\phi} ds + \sumk \int_0^t \Ltwoinner{\Tt_k \tB, \phi} d\widetilde{W}^k_s ,
\end{align*}
for all $\phi \in H$. 
\end{enumerate}
\end{Definition}

\begin{Definition}[Pathwise solution] Given $\Ff_0$-measurable initial data $B_0 \in L^2(\Om;H^s)$, we say that: \\
\begin{enumerate}
\item A pair $\left(B, \eta \right)$ is a local pathwise solution to the system \eqref{eq:EMHD_Ito} if $\eta$ is a strictly positive $\F$-stopping time and $B(\cdot \wedge \eta)$ is a progressively measurable stochastic process such that, $\P$-almost surely,
\[
B(\cdot \wedge \eta) \in L^2(\Om; C([0,\infty); H^s) \cap L^2([0,\infty);H^{s+\frac{\a}{2}})).
\]
Additionally, for every $t \geq 0$, the following identity holds:
\begin{align*}
&\Ltwoinner{B(t\wedge \eta),\phi} + \int_0^{t\wedge \eta} \Ltwoinner{\curl \bigl( (\curl B ) \times B \bigr) + \lb^\a B, \phi} ds \\
=\ &\Ltwoinner{B(0),\phi} + \haf \sumk \int_0^{t\wedge \eta} \Ltwoinner{\Tt^2_k \tB,\phi} ds + \sumk \int_0^{t\wedge \eta}  \Ltwoinner{\Tt_k \tB, \phi} d\widetilde{W}^k_s ,
\end{align*}
for all $\phi \in H$.\\
\item A triplet $(B,(\eta_k)_{k \geq 1},\xi)$ is a maximal pathwise solution if for each pair $(B,\eta_k)$ it holds that: \\
\begin{enumerate}[label=(\roman*)]
\item $(B,\eta_k)$ is a local pathwise solution. \\
\item $\eta_k$
 is an increasing sequence of stopping times such that $\lim_{k \to \infty} \eta_k = \xi$. \\
\item $\sup_{t \in [0,\eta_k]} \Hsnorm{B(t)} \geq k $ on the set $\{\xi < \infty \}$.
 \end{enumerate}
\end{enumerate}

\end{Definition}

\section{Existence of martingale solutions}\label{Sec:existence}
In this section, we prove the existence of martingale solutions by establishing uniform a priori estimates for the Galerkin approximations, followed by a compactness argument and passage to the limit.
\subsection{Energy estimates}
The Galerkin truncation of the system is given by 
\begin{subequations}\label{eq:EMHD_Ito_Galerkin}
\begin{align}
dB_n + \bigl( \chi^2_r \Pp_n \curl ( ( \curl B_n )\times B_n ) +\mu \lb^\a B_n \bigr) dt &= \haf \sum_{k =1}^\infty \Pp_n \Tt^2_k B_n dt + \sum_{k=1}^\infty \Pp_n \Tt_k B_n dW^k, \label{eq:EMHD_Ito_Galerkin_1} \\
B_n(0) &= \Pp_n B_0, \label{eq:EMHD_Ito_Galerkin_2}
\end{align}
\end{subequations}
where $\Tt_k u = c_k \cdot \nabla u$. Here $\Pp_n$ is the Leray projection defined by 
\[
\Pp_n B := \sum_{|k| \leq n} \hat{B}_k e^{2\pi i k \cdot x}.
\]

\begin{Proposition}[Uniform energy estimate]\label{prop:uni_energy_estimate}
Given $p \geq 2$, $T \geq 0$, $\a \in (1,2]$, $s > \frac{5}{2}$ and $B_0 \in L^\infty(\Omega,H^s)$ be $\Ff_0$-measurable, suppose that $B_n$ is the solution to the Galerkin approximated equations \eqref{eq:EMHD_Ito_Galerkin}, then there exists a universal constant $C := C(p,s,r,T)$ independent of $n$ such that
\begin{enumerate}[label=(\alph*)]
\item \label{prop:uni_energy_estimate_1} We have the uniform energy estimate in $n \in \N$:
\begin{equation*}
    \E \Bigl( \sup_{t \in [0,T]} \HsNorm{B_n(t)}^p +\int_0^T \Hsnorm{B_n(t)}^{p-2} \HsaNorm{B_n(t)}^2 dt \Bigr) \leq C \bigl( 1+ \E \Hsnorm{B_0}^p \bigr). 
\end{equation*}
\item \label{prop:uni_energy_estimate_2} For any $\gm \in \left[0,\haf \right)$ and any $\b \in (1,2]$, there is 
\[
\E \Big \lVert\int_0^\cdot \sumk \Pp_n \Tt_k B_n dW^k_t \Big \rVert^p_{W^{\gm,p}(0,T;H^{s-2+\frac{\b}{2}})} \leq C \bigl( 1+ \E \Hsnorm{B_0}^p \bigr ).
\]
\item \label{prop:uni_energy_estimate_3} It holds that 
\[
\E \Big \lVert B_n - \int_0^\cdot \sumk \Pp_n \Tt_k B_n dW^k_t \Big \rVert^2_{W^{1,2}(0,T;H^{s-2+\frac{\a}{2}})} \leq C \bigl(1+ \E \Hsnorm{B_0}^4\bigr).
\]
\item  \label{prop:uni_energy_estimate_4} In addition, we have that for $\gm \in \left[0,\haf \right)$:
\[
\E \left \lVert B_n\right \rVert^2_{W^{\gm,2}(0,T;H^{s-2+\frac{\a}{2}})} \leq C \bigl(1+ \E \Hsnorm{B_0}^4\bigr).
\]
\end{enumerate}
\end{Proposition}

\begin{Remark}
It is possible to extend the similar results towards the Sobolev space with lower regularity $s \in \left( \frac{3}{2}, \frac{5}{2} \right]$ as well as the optimal one, which requires a finer analysis; see \cite{WZ17,Dai21}. 
\end{Remark}

\begin{proof}

By Lemma \ref{lemma:Ito_formula_EMHD}, we have 
\begin{align*}
d \twonorm{\lb^s B_n}^p &= -p \left< \chir^2 \Pp_n \bigl( \curl ( (\curl B_n) \times B_n ) + \mu \lb^\a B_n \bigr), \lb^{2s}B_n \right> \twonorm{\lb^s B_n}^{p-2} dt \notag \\
&+ \frac{p}{2} \sum_{k =1}^\infty \bigl ( \<\Pp_n \Tt^2_k B_n ,\lb^{2s} B_n\> + \twonorm{\lb^s \Pp_n \Tt_k B_n}^2 \bigr) \twonorm{\lb^s B_n}^{p-2} dt \notag \\
&+\haf p (p-2) \twonorm{\lb^s B_n}^{p-4} \sum_{k=1}^\infty \< \Pp_n \Tt_k B_n, \lb^{2s} B_n  \>^2 dt \notag \\
&+ p \twonorm{\lb^s B_n}^{p-2} \sum_{k =1 }^\infty \< \Pp_n \Tt_k B_n, \lb^{2s} B_n \> dW^k  \notag \\
&=: I_1 dt +I_2 dt +I_3 dt +I_4 d\W .
\end{align*}
Firstly, we consider 
\[
\int_0^t I_2 d\tau = \frac{p}{2} \sumk \int_0^t \bigl ( \<\Pp_n \Tt^2_k B_n ,\lb^{2s} B_n\> + \twonorm{\lb^s \Pp_n \Tt_k B_n}^2 \bigr) \twonorm{\lb^s B_n}^{p-2} d\tau.
\]
Using the fact that $c_k$ is divergence-free for all $k$, the integrand can be written by
\begin{align*}
&\left <\Pp_n \Tt^2_k B_n , \lb^{2s} B_n \right> + \twonorm{\lb^s \Pp_n \Tt_k B_n}^2 = \Ltwoinner{\Tt_k \Tt_k B_n,\lb^{2s}B_n} + \Ltwoinner{\lb^s\Tt_k B_n, \lb^s \Tt_k B_n} \\
\leq&\Ltwoinner{\lb^s \Tt_k \Tt_k B_n,\lb^s B_n} + \Ltwoinner{\lb^s \Tt_k B_n, \lb^s \Tt_k B_n} \\
=&\Ltwoinner{[\lb^s,\Tt_k]\Tt_k B_n,\lb^s B_n} + \Ltwoinner{\Tt_k \lb^s \Tt_k B_n, \lb^s B_n} + \Ltwoinner{\lb^s \Tt_k B_n, \lb^s \Tt_k B_n} \\
=&\Ltwoinner{\left[ [\lb^s,\Tt_k], \Tt_k \right]B_n,\lb^s B_n} + \Ltwoinner{T_k[\lb^s,\Tt_k]B_n,\lb^s B_n}+ \Ltwoinner{\Tt_k \lb^s \Tt_k B_n, \lb^s B_n} + \Ltwoinner{\lb^s \Tt_k B_n, \lb^s \Tt_k B_n} \\
= &\Ltwoinner{\left[ [\lb^s,\Tt_k],\Tt_k \right]B_n, \lb^s B_n}- \Ltwoinner{[\lb^s,\Tt_k]B_n,\Tt_k \lb^s B_n}- \Ltwoinner{\lb^s \Tt_k B_n,\Tt_k \lb^s B_n} + \Ltwoinner{\lb^s\Tt_k B_n,\lb^s \Tt_k B_n} \\
= & \Ltwoinner{\left[[\lb^s, \Tt_k],\Tt_k \right]B_n, \lb^s B_n} - \Ltwoinner{[\lb^s,\Tt_k]B_n,\Tt_k \lb^s B_n} + \Ltwoinner{\lb^s \Tt_k B_n,[\lb^s,\Tt_k]B_n} \\
= &\Ltwoinner{\left[\lb^s,\Tt_k \right]B_n,\lb^s B_n} + \Ltwoinner{[\lb^s,\Tt_k]B_n,[\lb^s,\Tt_k]B_n},
\end{align*}
where we dropped the Leray projection $\Pp_n$ since $\twonorm{\Pp_n B_n} \leq \twonorm{B_n}$. Otherwise, we can replace $B_n$ by $\Pp_n B_n$ in the above and obtain the same estimate. Now, by Lemmas \ref{lemma:commutator_est_1} and \ref{lemma:commutator_est_2}, we have that 
\begin{align*}
\Ltwoinner{\left[\lb^s,\Tt_k \right]B_n,\lb^s B_n} &\lesssim \HssNorm{c_k} \Hsnorm{B_n}^2, \\
\Ltwoinner{[\lb^s,\Tt_k]B_n,[\lb^s,\Tt_k]B_n} &\lesssim \HsNorm{c_k}^2 \Hsnorm{B}^2,
\end{align*}
whence,
\begin{equation}\label{proof:uniform_energy_estimate_1}
\left <\Pp_n \Tt^2_k B_n , \lb^{2s} B_n \right> + \twonorm{\lb^s \Pp_n \Tt_k B_n}^2 \les \ltwoHssnorm{c}^2 \Hsnorm{B_n}^2,
\end{equation}
and therefore 
\begin{equation}\label{proof:uniform_energy_estimate_2}
\int_0^t I_2 d\tau \les \ltwoHssnorm{c}^2 \int_0^t \Hsnorm{B_n}^p d\tau.
\end{equation}
Next we estimate $I_3$. Recall that 
\begin{align*}
I_3 &= \haf p(p-2) \int_0^t \twonorm{\lb^s B_n}^{p-4} \sumk \Ltwoinner{\Pp_n \Tt_k B_n,\lb^{2s}B_n}^2 d\tau. 
\end{align*}
To bound this term, we first note that
\begin{align*}
\Ltwoinner{\Pp_n \Tt_k B_n, \lb^{2s}B_n} &= \Ltwoinner{\Tt_k B_n,\lb^{2s}B_n} = \Ltwoinner{\lb^s \Tt_k B_n,\lb^s B_n} \\
&= \Ltwoinner{\lb^s \Tt_k B_n, \lb^s B_n} -\Ltwoinner{\Tt_k \lb^s B_n,\lb^s B_n}.
\end{align*}
Since $c_k$ is divergence free for each $k$, we have
\[
\Ltwoinner{\Tt_k \lb^s B_n,\lb^s B_n} = \int_{\T^3} c_k \cdot \nabla (\lb^s B_n) \cdot \lb^s B_n dx = \int_{\T^3} \nabla \cdot c_k |\lb^s B_n|^2 \equiv 0.
\]
Therefore,
\begin{align} \label{proof:uniform_energy_estimate_3}
\Ltwoinner{\Pp_n \Tt_k B_n, \lb^{2s}B_n} = \Ltwoinner{[\lb^s,\Tt_k]B_n, \lb^s B_n} \leq \HssNorm{c_k} \HsNorm{B_n}^2.
\end{align}
Hence we arrive at,
\begin{align}
\int_0^t I_3 d\tau &= \haf p (p-2) \int_0^t \twonorm{\lb^s B_n}^{p-4} \sumk \Ltwoinner{\Pp_n \Tt_k B_n,\lb^{2s}B_n}^2 d \tau \notag \\
&\leq \haf p(p-2) \int_0^t \HsNorm{B_n}^{p-4} \biggl ( \sumk \HssNorm{c_k}^2 \biggr) \Hsnorm{B_n}^4 d \tau \notag \\
&\les_p \ltwoHssnorm{c}^2 \int_0^t \HsNorm{B_n}^p d \tau. \label{proof:uniform_energy_estimate_4}
\end{align}
For the estimate of $I_4$, applying Burkholder-Davis-Gundy Inequality (Theorem \ref{Thm:Burkholder-Davis-Gundy Inequality}) and \eqref{proof:uniform_energy_estimate_3} yields
\begin{align}
\E \biggl( \sup_{\tau \in [0,t]} \Bigl| \int_0^\tau I_4 d\W  \Bigr| \biggr) &= p \E \sup_{\tau \in [0,t]} \biggl( \Bigl|\int_0^\tau \twonorm{\lb^s B_n}^{p-2} \sumk \Ltwoinner{\Pp_n \Tt_k B_n , \lb^{2s} B_n} dW^k \Bigr| \biggr) \notag \\
&\leq C_p \E \biggl( \int_0^t \Hsnorm{B_n}^{2(p-2)} \sumk \Ltwoinner{\Pp_n \Tt_k B_n,\lb^{2s}B_n}^2 d\tau  \biggr)^\haf \notag \\
&\leq C_p \E \biggl( \int_0^t \Hsnorm{B_n}^{2(p-2)} \biggl( \sumk \HssNorm{c_k}^2 \biggr) \Hsnorm{B_n}^4 d\tau  \biggr)^\haf \notag \\
&\leq C_p \ltwoHssnorm{c} \E \biggl (\int_0^t \Hsnorm{B_n}^{2p}d\tau \biggr)^\haf \notag \\
&\leq \frac{1}{2} \E \Bigl (\sup_{\tau \in [0,t]} \Hsnorm{B_n}^p\Bigr) + C_p \ltwoHssnorm{c}^2 \E \biggl(\int_0^t \Hsnorm{B_n}^p d\tau \biggr). \label{proof:uniform_energy_estimate_5}
\end{align}
To estimate $I_1$, we use integration by parts to obtain
\begin{align*}
\int_0^t I_1 d\tau &= -p \int_0^t \Ltwoinner{\chi^2_r \Pp_n \left( \curl (\curl B_n)\times B_n\right) + \mu \lb^\a B_n, \lb^{2s}B_n} \twonorm{\lb^s B_n}^{p-2} d\tau \\
&= p \int_0^t \Ltwoinner{\chi^2_r \Pp_n \lb^s \left( B_n \times \curl B_n \right), \lb^{s}(\curl B_n) }  \Hsnorm{B_n}^{p-2} d\tau -\mu p\int_0^t \HsaNorm{B_n}^2 \Hsnorm{B_n}^{p-2} d\tau.
\end{align*}
By Lemma \ref{lemma:est_hall_term} given below, we have that 
\begin{align}\label{proof:uniform_energy_estimate_6}
\int_0^t I_1 d\tau \leq C_r p \int_0^t \Hsnorm{B_n}^p d\tau - \frac{\mu p}{2} \int_0^t \Hsnorm{B_n}^{p-2} \HsaNorm{B_n}^2 d\tau.
\end{align}


\noindent Collecting \eqref{proof:uniform_energy_estimate_2}, \eqref{proof:uniform_energy_estimate_4}, \eqref{proof:uniform_energy_estimate_5} and \eqref{proof:uniform_energy_estimate_6} gives
\begin{align*}
&\E \sup_{\tau \in [0,t]} \Hsnorm{B_n}^p + \E \int_0^t \HsaNorm{B_n}^2 \Hsnorm{B_n}^{p-2} d \tau \\
\les_{s,p,r} &\ltwoHssnorm{c}^2 \Bigl( \E \Hsnorm{B_0}^p + \E \int_0^t \Hsnorm{B_n}^p d\tau \Bigr).
\end{align*}
Applying Gr\"onwall's inequality yields
\[
\E \Bigl( \sup_{\tau \in [0,T]} \Hsnorm{B_n}^p + \int_0^T \HsaNorm{B_n}^2 \Hsnorm{B_n}^{p-2} d\tau  \Bigr ) \les_{s,p,r,c,T} 1 +\E \Hsnorm{B_0}^p,
\]
which establishes \ref{prop:uni_energy_estimate_1}. 

\noindent Now we will show \ref{prop:uni_energy_estimate_2}. Using Theorem \ref{Thm:Burkholder-Davis-Gundy Inequality} with $\gm \in \left[0,\haf \right)$ and $p \geq 2$ gives 
\begin{align*}
    \E \Big \lVert \int_0^\cdot \sumk \Pp_n \Tt_k B_n dW^k_t \Big \rVert^p_{W^{\gm,p}(0,T;H^{s-2+\frac{\b}{2}})} &\leq C_p \E \int_0^T \biggl( \sumk \|\Tt_k B_n \|^2_{H^{s-1}(\T^3)} \biggr)^{\frac{p}{2}} dt\\
    &\leq C_p T \|c\|^p_{\ell(\N;H^{s-1}(\T^3))} \E \sup_{t \in [0,T]} \Hsnorm{B_n(t)}^p \\
    &\leq C \bigl( 1 + \E \Hsnorm{B_0}^p \bigr),
\end{align*}
in which we used $s - 1 > \frac{3}{2}$, $\b \leq 2$ and \ref{prop:uni_energy_estimate_1}. Therefore we have shown \ref{prop:uni_energy_estimate_2}. 
\vspace{1em} \\
For the proof of \ref{prop:uni_energy_estimate_3}, firstly, it follows from \eqref{eq:EMHD_Ito_Galerkin} that
\begin{align*}
    &B_n(t) - \int_0^t \Pp_n \Tt_k B_n dW^k_\tau \\
    = \  &\Pp_n B_0 - \int_0^t \bigl( \chi_r^2 \Pp_n \curl (\curl B_n ) \times B_n + \mu \lb^\a B_n \bigr) d\tau + \haf \int_0^t \sumk \Pp_n \Tt^2_k B_n d \tau. 
\end{align*}
Again, since $s-1> \frac{3}{2}$ and $1< \a \leq 2$, we have 
\begin{align*}
 &\| \chi^2_r \Pp_n \curl \bigl( (\curl B_n)\times B_n \bigr) + \mu \lb^\a B_n  \|^2_{H^{s-2+\frac{\a}{2}}(\T^3)} \\
 \les &\ \|\grad B_n\|^2_{H^{s-1+\frac{\a}{2}}(\T^3)} \|B_n\|^2_{H^{s-1+\frac{\a}{2}}(\T^3)} + \mu \|B_n\|^2_{H^{s+\frac{\a}{2}-2+\frac{\a}{2}}(\T^3)} \les  \bigl( 1+ \HsaNorm{B_n}^2 \bigr)\Hsnorm{B_n}^2,
\end{align*}
and that 
\begin{align*}
\|\Pp_n \Tt^2_k B_n\|_{H^{s-2+\frac{\a}{2}}(\T^3)} &\les_s \|c_k\|_{W^{s-2+\frac{\a}{2},\infty}(\T^3)} \|\Tt_k B_n\|_{H^{s-1+\frac{\a}{2}}(\T^3)} \\
&\les_s \|c_k \|_{H^{s-1+\a}(\T^3)} \|c_k\|_{H^{s-\haf+\a}(\T^3)} \HsaNorm{B_n} \les \HssNorm{c_k}^2 \HsaNorm{B_n}.
\end{align*}
Therefore,
\begin{align*}
&\E \Big \lVert B_n - \int_0^\cdot \sumk \Pp_n \Tt_k B_n dW^k_t \Big \rVert^2_{W^{1,2}(0,T;H^{s-2+\frac{\a}{2}})} \\
\leq \  &C \E \Bigl ( 1 + \Hsnorm{B_0}^2 + \sup_{t \in [0,T]} \Hsnorm{B_n}^2 + \int_0^T \HssNorm{B_n}^2 \Hsnorm{B_n}^2 dt \Bigr).
\end{align*}
The bound \ref{prop:uni_energy_estimate_3} then follows from \ref{prop:uni_energy_estimate_1}
and the bound \ref{prop:uni_energy_estimate_4} follows from \ref{prop:uni_energy_estimate_2} and \ref{prop:uni_energy_estimate_3} by taking suitable values of $p$, hence we conclude the proof.
\end{proof}
\noindent We now provide the estimate for the Hall term used in the proof of Proposition~\ref{prop:uni_energy_estimate}.
\begin{Lemma}\label{lemma:est_hall_term}
The nonlinear hall term in Proposition \ref{prop:uni_energy_estimate} satisfies the following estimate: 
\begin{align}
&\int_0^t \Ltwoinner{\chi_r^2 \Pp_n \lb^s \left( B_n \times \curl B_n \right), \lb^s \left( \curl B_n \right)} \Hsnorm{B_n}^{p-2} d\tau \notag \\
\leq &\  C_r \int_0^t \Hsnorm{B_n}^p d\tau + \frac{\mu}{2} \int_0^t \Hsnorm{B_n}^{p-2} \HsaNorm{B_n}^2 d\tau.    \label{lemma:est_hall_term_1}
\end{align}
\end{Lemma}

\begin{proof}
We follow the argument by \cite{CWW15}. First, for fixed $n \in \N$,
we write
\begin{align*}
&\int_0^t \Ltwoinner{\chi_r^2 \Pp_n \lb^s \left( B_n \times \curl B_n \right), \lb^s \left( \curl B_n \right)} \Hsnorm{B_n}^{p-2} d\tau \\
=\ &\sum_{q = -1}^\infty \int_0^t \chi_r^2 \Ltwoinner{ \Pp^2_n \Dq^2 \lb^s \left( B_n \times \curl B_n \right), \lb^s \left( \curl B_n \right)} \Hsnorm{B_n}^{p-2} d\tau \\
=\ &\sum_{q = -1}^\infty \int_0^t \chi_r^2 \Ltwoinner{  \Pp_n \Dq \lb^s \left( B_n \times \curl B_n \right), \Pp_n \Dq \lb^s \left( \curl B_n \right)} \Hsnorm{B_n}^{p-2} d\tau.
\end{align*}
Next, we use the fact that 
\[
\int_0^t \chi_r^2\Ltwoinner{\bigl(B_n \times \Pp_n \Dq \lb^s (\curl B_n) \bigr), \Pp_n \Dq \lb^s \left( \curl B_n \right)} \Hsnorm{B_n}^{p-2} d\tau  = 0,
\]
to get 
\begin{align*}
&\int_0^t \chi_r^2 \Ltwoinner{  \Pp_n \Dq \lb^s \left( B_n \times \curl B_n \right), \Pp_n \Dq \lb^s \left( \curl B_n \right)} \Hsnorm{B_n}^{p-2} d\tau \\
= \ &\int_0^t \chi_r^2\Ltwoinner{ \Pp_n \Dq \lb^s \left( B_n \times \curl B_n \right) - \bigl(B_n \times \Pp_n \Dq \lb^s (\curl B_n) \bigr), \Pp_n \Dq \lb^s \left( \curl B_n \right)} \Hsnorm{B_n}^{p-2} d\tau \\
= \ &\int_0^t \chi_r^2\Ltwoinner{ [\Pp_n \Dq \lb^s, B_n \cdot \nabla ]B_n, \Pp_n \Dq \lb^s \left( \curl B_n \right)} \Hsnorm{B_n}^{p-2} d\tau \\
+ \ &\int_0^t \chi_r^2\Ltwoinner{  \Pp_n \Dq \lb^s \haf(\grad |B_n|^2) - \grad (\Pp_n \Dq \lb^s B_n) \cdot B_n, \Pp_n \Dq \lb^s \left( \curl B_n \right)} \Hsnorm{B_n}^{p-2} d\tau \\
=: \ & J_1 + J_2,
\end{align*}
for each $q \geq -1$, where we make use of the vector calculus identity
\begin{equation}\label{equ:vector_calculus_identity}
A \times (\curl B) = (\grad B) \cdot A - (A \cdot \grad )B.
\end{equation}
Invoking Bony's paraproduct formula, we decompose
\[
J_1 =: J_{11} + J_{12} +J_{13},\quad\text{where}
\]
\begin{align*}
J_{11} &= \int_0^t \chi^2_r \biggl( \sum_{|q-l|\leq 2} \Ltwoinner{ \left[ \Pp_n \Dq \lb^s, \Ss_{l-2} B_n \cdot \grad \right] \Dl B_n, \Pp_n \lb^s \Dq (\curl B_n) } \biggr ) \Hsnorm{B_n}^{p-2} d\tau, \\
J_{12} &= \int_0^t \chi^2_r \biggl( \sum_{|q-l|\leq 2} \Ltwoinner{ \left[ \Pp_n \Dq \lb^s, \Dl B_n \cdot \grad \right] \Ss_{l-2} B_n, \Pp_n \lb^s \Dq (\curl B_n) } \biggr ) \Hsnorm{B_n}^{p-2} d\tau, \\
J_{13} &= \int_0^t \chi^2_r \biggl( \sum_{|q-l|\leq 2} \Ltwoinner{ \left[ \Pp_n \Dq \lb^s, \Dl B_n \cdot \grad \right] \wtDl B_n, \Pp_n \lb^s \Dq (\curl B_n) } \biggr ) \Hsnorm{B_n}^{p-2} d\tau.
\end{align*}
Here we adapt the notation 
\[
\left [ \Dq, \Ss_{l-2} u \cdot \nabla \right]v := \Dq \left( \Ss_{l-2}u \cdot \nabla \right)v - \Ss_{l-2}u \cdot \nabla \Dq v.
\]
We now apply Lemma~\ref{lemma:commutator_est_3} together with H\"older's inequality and Bernstein's inequality to obtain
\begin{align*}
    J_{11} &\les \int_0^t \chi_r^2 \Bigl( \l^{2s}_q \twonorm{\Dq \curl B_n}  \sum_{|q-l|\leq 2} \LinfNorm{\grad \Ss_{l-2} B_n} \twonorm{\Dl B_n} \Bigr) \Hsnorm{B_n}^{p-2} d\tau \\
    &\les \int_0^t \chi_r^2 \left( \l^{2s}_q \twonorm{\Dq \curl B_n}  \LinfNorm{\grad \Ss_{q-2} B_n} \twonorm{\Dq B_n} \right) \Hsnorm{B_n}^{p-2} d\tau\\
    &\les \int_0^t \chi_r^2 \l^{1+2s}_q \Linfnorm{\grad B_n} \twonorm{\Dq B_n}^2 \Hsnorm{B_n}^{p-2} d\tau \\
    &\les r \int_0^t \l^{1+2s}_q \twonorm{\Dq B_n}^2 \Hsnorm{B_n}^{p-2} d\tau, 
\end{align*}
and
\begin{align*}
J_{12} &\les \int_0^t \chi_r^2 \l^{2s}_q \biggl( \sum_{|q-l|\leq 2} \twonorm{ \Dl B_n} \Linfnorm{\grad \Ss_{l-2}B_n} \biggr) \twonorm{\Dq \curl B_n} \Hsnorm{B_n}^{p-2} d\tau \\
&\les r \int_0^t \chi_r^2 \l^{1+2s}_q \Linfnorm{\grad B_n} \twonorm{\Dq B_n}^2 d\tau \les r \int_0^t \l^{1+2s}_q \twonorm{\Dq B_n}^2  \Hsnorm{B_n}^{p-2} d\tau.
\end{align*}
Similarly,
\begin{align*}
J_{13} &\les \int_0^t \chi_r^2 \l^{1+2s}_q \twonorm{\Dq B_n} \biggl( \sum_{l \geq q-2} \Linfnorm{\grad \Dq \Dl B_n} \twonorm{\Dq \wtDl B_n} \biggr) \Hsnorm{B_n}^{p-2} d\tau \\
&\les \int_0^t \chi_r^2 \l^{1+2s}_q \twonorm{\Dq B_n} \Linfnorm{\grad B_n} \biggl( \sum_{l \geq q-2} \l_{q-l}  \twonorm{\Dl B_n} \biggr) \Hsnorm{B_n}^{p-2} d \tau  \\
&\les r \int_0^t \l_q^{1+2s} \twonorm{\Dq B_n} \biggl( \sum_{l \geq q-2} \l_{q-l}  \twonorm{\Dl B_n} \biggr) \Hsnorm{B_n}^{p-2} d \tau. 
\end{align*}
Note that we only use H\"older's inequality instead of the commutator estimate for $J_{12}$ above. Therefore,  
\begin{equation}\label{proof_lemma_est_hall_term_1}
J_1 \les r \int_0^t \l^{1+2s}_q \twonorm{\Dq B_n} \Bigl( \twonorm{\Dq B_n} + \sum_{l \geq q -2} \l_{q-l} \twonorm{\Dl B_n} \Bigr) \Hsnorm{B_n}^{p-2} d\tau.
\end{equation}
Now we consider $J_2$. In a similar way, we decompose $J_2$ into
\[
J_2 =: J_{21} + J_{22} + J_{23},\quad \text{where}
\]
\begin{align*}
J_{21} &= \int_0^t \chi^2_r \biggl( \sum_{|q-l|\leq 2} \Ltwoinner{ \left[ \Pp_n \Dq \lb^s, \Dl B_n \right]\grad \Ss_{l-2} B_n , \Pp_n \lb^s \Dq (\curl B_n) } \biggr ) \Hsnorm{B_n}^{p-2} d\tau, \\
J_{22} &= \int_0^t \chi^2_r \biggl( \sum_{|q-l|\leq 2} \Ltwoinner{ \left[ \Pp_n \Dq \lb^s, \Ss_{l-2} B_n\right]\grad \Dl B_n  , \Pp_n \lb^s \Dq (\curl B_n) } \biggr ) \Hsnorm{B_n}^{p-2} d\tau, \\
J_{23} &= \int_0^t \chi^2_r \biggl( \sum_{l \geq q-2} \Ltwoinner{ \Pp_n \lb^s\Dq  \grad \left( \haf \Dl B_n \cdot \wtDl B_n \right) , \Pp_n \lb^s \Dq (\curl B_n) } \biggr ) \Hsnorm{B_n}^{p-2} d\tau \\
&- \int_0^t \chi_r^2 \biggl( \sum_{l \geq q-2} \Ltwoinner{\left(\grad \Pp_n \lb^s \Dq \Dl B_n \right) \cdot \wtDl B_n, \Pp_n \lb^s \Dq (\curl B_n) } \biggr) \Hsnorm{B_n}^{p-2} d\tau.
\end{align*}
We bound $J_{21}$ by H\"older's inequality,
\begin{align*}
    J_{21} &\les \int_0^t \chi_r^2 \biggl( \sum_{|q-l|\leq 2} \l^{2s}_q \Linfnorm{\grad \Ss_{l-2} B_n} \twonorm{\Dl B_n} \twonorm{\Dq(\curl B_n)} \biggr) \Hsnorm{B_n}^{p-2} d\tau, \\
    &\les \int_0^t \chi_r^2 \l^{1+2s}_q \Linfnorm{\grad B_n} \twonorm{\Dq B_n}^2 \Hsnorm{B_n}^{p-2} d\tau, \\
    &\les r \int_0^t \l^{1+2s}_q  \twonorm{\Dq B_n}^2 \Hsnorm{B_n}^{p-2} d\tau .
\end{align*}
For $J_{22}$, we use a similar commutator estimate as in Lemma \ref{lemma:commutator_est_3} and estimate as 
\begin{align*}
    J_{22} &\les  \int_0^t \chi_r^2 \biggl( \sum_{|q-l|\leq 2} \l^{2s}_q \Linfnorm{\grad \Ss_{q-2} B_n} \twonorm{\Dq B_n} \twonorm{\Dq (\curl B_n)} \biggr) \Hsnorm{B_n}^{p-2} d\tau,  \\
    &\les \int_0^t \chi_r^2 \left(  \l_q^{1+2s} \Linfnorm{\grad B_n} \twonorm{\Dq B_n}^2 \right) \Hsnorm{B_n}^{p-2} d\tau, \\
    &\les r \int_0^t  \l_q^{1+2s} \twonorm{\Dq B_n}^2  \Hsnorm{B_n}^{p-2} d\tau. 
\end{align*}
We apply Bernstein's inequality to $J_{23}$ and conclude that 
\[
J_{23} \les r \int_0^t \l^{1+2s}_q  \twonorm{\Dq B_n} \biggl( \sum_{l \geq q-2} \l_{q-l} \twonorm{\Dl B_n} \biggr) \Hsnorm{B_n}^{p-2} d\tau,
\]
therefore, we have 
\begin{equation}\label{proof_lemma_est_hall_term_2}
J_2 \les \int_0^t \l^{1+2s}_q \twonorm{\Dq B_n} \Bigl( \twonorm{\Dq B_n} + \sum_{l \geq q-2} \l_{q-l} \twonorm{\Dl B_n}  \Bigr) \Hsnorm{B_n}^{p-2} d\tau. 
\end{equation}
Putting together \eqref{proof_lemma_est_hall_term_1} and \eqref{proof_lemma_est_hall_term_2} brings
\begin{align*}
J_1 + J_2 &\les r\int_0^t \l^{1+2s}_q \twonorm{\Dq B_n} \Bigl( \twonorm{\Dq B_n} + \sum_{l \geq q-2} \l_{q-l} \twonorm{\Dl B_n}  \Bigr) \Hsnorm{B_n}^{p-2} d\tau \\ 
&\les r\int_0^t \l^{1+2s}_q \twonorm{\Dq B_n}^2 \Hsnorm{B_n}^{p-2} d\tau  \\
&\quad + r \int_0^t \l_q^{1+2s}\biggl( \sum_{l \geq q-2} \l_{q-l}  \twonorm{\Dl B_n} \biggr)^2 \Hsnorm{B_n}^{p-2} d\tau. 
\end{align*}
Notice that if we sum over $q \geq -1$ of the right hand side of the above, then
it follows that 
\begin{align*}
    &r\sum_{q=-1}^\infty \int_0^t \l^{1+2s}_q \twonorm{\Dq B_n}^2 \Hsnorm{B_n}^{p-2} d\tau, \\
    &\quad \les r \int_0^t \sum_{q=-1}^\infty \left( \l^{2s}_q \twonorm{\Dq B_n}^2 \right)^\haf \left( \l_q^{2(s+\frac{\a}{2})}\twonorm{\Dq B_n}^2  \right)^\haf   \Hsnorm{B_n}^{p-2} d\tau, \\
    &\quad \les r \int_0^t \biggl( \sum_{q=-1}^\infty \l^{2s}_q \twonorm{\Dq B_n}^2 \biggr)^\haf  \biggl(\sum_{q=-1}^\infty \l_q^{2(s+\frac{\a}{2})}\twonorm{\Dq B_n}^2  \biggr)^\haf   \Hsnorm{B_n}^{p-2} d\tau, \\
    &\quad \les r \int_0^t \Hsnorm{B_n} \HsaNorm{B_n} \Hsnorm{B_n}^{p-2} d\tau, \\
    &\quad \leq C r^2 \int_0^t \Hsnorm{B_n}^p d\tau + \frac{\mu}{4} \int_0^t \HsaNorm{B_n}^2 \Hsnorm{B_n}^{p-2} d\tau,
    \end{align*}
and that
\begin{align*}
&r \int_0^t \l_q^{1+2s}\biggl( \sum_{l \geq q-2} \l_{q-l}  \twonorm{\Dl B_n} \biggr)^2 \Hsnorm{B_n}^{p-2} d\tau \\
&\quad \les r \int_0^t \sum_{q=-1}^\infty \biggl(\sum_{l \geq q-2} \l_{q-l} \l^{s +\haf}_q  \twonorm{\Dl B_n} \biggr)^2 \Hsnorm{B_n}^{p-2} d\tau \\
&\quad \les r \int_0^t \|B_n\|^2_{H^{s+\haf}(\T^3)}\Hsnorm{B_n}^{p-2} d\tau \les r \int_0^t \Hsnorm{B_n}^{\frac{2}{\a}(\a-1)} \HsaNorm{B_n}^{\frac{2}{\a}} \Hsnorm{B_n}^{p-2} d\tau \\
&\quad \leq Cr^{1+\frac{1}{\a-1}} \int_0^t \Hsnorm{B_n}^p d\tau + \frac{\mu}{4} \int_0^t \HsaNorm{B_n}^2 \Hsnorm{B_n}^{p-2} d\tau.  
\end{align*}
Finally, recall that 
\begin{align*}
&\int_0^t \Ltwoinner{\chi_r^2 \Pp_n \lb^s \left( B_n \times \curl B_n \right), \lb^s \left( \curl B_n \right)} \Hsnorm{B_n}^{p-2} d\tau, \\
=\ &\sum_{q = -1}^\infty \int_0^t \chi_r^2 \Ltwoinner{  \Pp_n \Dq \lb^s \left( B_n \times \curl B_n \right), \Pp_n \Dq \lb^s \left( \curl B_n \right)} \Hsnorm{B_n}^{p-2} d\tau \les \sum_{q = -1}^\infty (J_1 + J_2) \\
\leq\ &C r^{1+\frac{1}{\a-1}} \int_0^t \Hsnorm{B_n}^p d\tau + \frac{\mu}{2} \int_0^t \HsaNorm{B_n}^2 \Hsnorm{B_n}^{p-2} d\tau,
 \end{align*}
and thus we conclude the proof.
\end{proof}

\subsection{Existence} Let $\Ss = \left(\Om,\Ff,\F,\P \right)$, define the path space 
\[
U := H^s \times L^2(0,T;H^s) \cap C([0,T];H^1) \times C ([0,T];\mathscr{U}).
\]  
Given any initial data $B_0 \in L^\infty(\Om;H^s)$, let $\mu^n_0$ be the law of $\Pp_n B_0$, $\mu^n_B$ be the law of $B_n$ and $\mu_{\W}$ be the law of Wiener process on $C([0,T];\mathscr{U})$. Define 
\[
\mu^n := \mu_0^n \otimes \mu^n_B \otimes \mu_\W 
\]
be the law on $U$. Then we have the existence of the martingale solutions given by the following proposition.
\begin{Proposition}\label{prop:martingale_solution}
Given the same conditions as in Proposition \ref{prop:uni_energy_estimate}, there exists a martingale solution $\bigl(\tB_0, \tB, \widetilde{\W}, \widetilde{\Ss} \bigr)$ to the system~\eqref{eq:EMHD_cutoff_Ito_1} - \eqref{eq:EMHD_cutoff_Ito_3} on $\left[ 0, T \right]$, where $\widetilde{\Ss} = \bigl(\widetilde{\Om},\tFf, \tF,\tP \bigr)$. Moreover, $\widetilde{B}$ satisfies that 
\begin{equation}\label{prop:martingale_solution_1}
\tB \in L^2 \bigl( \widetilde{\Om}; C (\left[0,T\right];H^s)\bigr) \cap  L^2 \bigl( \widetilde{\Om}; L^2 (\left[0,T\right];H^{s+1})\bigr).
\end{equation}

\end{Proposition}

\noindent Before proving the above, we state the following lemma due to \cite{ANR11}, which ensure the convergence of the stochastic terms in the Galerkin system. For the proof, see Appendix \ref{appendix:proof_lemma_stochastic_convergence}.
\begin{Lemma} \label{Lemma:stochastic_convergence}
Let $X$ be a separable Hilbert space and $\left( \Om, \Ff, \P \right)$ be a probability space. Consider a sequence of stochastic bases $\Ss^n = \left( \Om, \Ff, \F^n, \P \right)$ and $\W^n$ a sequence of $\F$-adapted cylindrical Wiener process with reproducing kernel Hilbert space $\mathscr{U}$. Suppose that the sequence $\{G^n\}_{n \in \N}$ of $X$-valued $\F^n$ predictable processes satisfy that $G^n \in L^2([0,T];L_2(\mathscr{U},X))$ almost surely. In addition, suppose also that we have a stochastic basis $\Ss = \left( \Om, \Ff, \F, \P \right)$ and an $\F$ predictable process $G \in L^2 ([0,T];L_2(\mathscr{U},X))$. Then if 
\begin{equation*}
W^n \to W \ \ \text{in probability in } C([0,T];\mathscr{U}) \ \ \text{and} \quad
G^n \to G \ \ \text{in probability in } C([0,T];L_2(\mathscr{U},X)),
\end{equation*}
it holds that 
$\displaystyle \int_0^t G^n dW^n \to \int_0^t G dW  \ \ \text{in probability in } C([0,T];X).$
\end{Lemma}
\begin{proof}[Proof of Proposition \ref{prop:martingale_solution}]
We divide the proof into two parts. First, we use a compactness argument to obtain a random variable in $U$ as the limit of the Galerkin solutions. Second, we verify that this limit is a martingale solution to \eqref{eq:EMHD_Ito_Galerkin}.

\smallskip
\noindent \textbf{Step 1: Compactness}. By Theorem \ref{Thm:compactness_thm_1}, we have the compact embedding: $$L^2(0,T;H^{s+\frac{\a}{2}}) \cap W^{\frac{1}{4},2}(0,T;H^{s-2+\frac{\a}{2}}) \subset \subset L^2(0,T;H^s).
$$ Furthermore, choosing $\gm \in (0,\haf)$ and $p \in (1,\infty)$ such that $\gm p >1$, it follows from Theorem \ref{Thm:compactness_thm_2} that the embeddings $    W^{1,2}(0,T;H^{s-2+\frac{\a}{2}}) \subset \subset C([0,T];H^{1}),  \;
    W^{\gm,p}(0,T;H^{s-2+\frac{\a}{2}}) \subset \subset C([0,T];H^{1})$
are both compact. To prove tightness of the family of measures $\{\mu^n\}_{n\in\mathbb N}$, we consider the following balls
\begin{align*}
V^1_R &= \Bigl \{ B \in L^2(0,T;H^{s+\frac{\a}{2}}) \cap W^{\frac{1}{4},2}(0,T;H^{s-2+\frac{\a}{2}}) : \| B\|^2_{L^2(0,T;H^{s+\frac{\a}{2}})} + \|B\|^2_{W^{\frac{1}{4},2}(0,T;H^{s-2+\frac{\a}{2}})} \leq R^2  \Bigr \}, \\
V^{2,1}_R &= \Bigl \{ B \in W^{1,2}(0,T;H^{s-2+\frac{\a}{2}}) : \|B\|^2_{W^{1,2}(0,T;H^{s-2+\frac{\a}{2}})} \leq R^2  \Bigr \}, \\
V^{2,2}_R &= \Bigl \{ B \in W^{\gm,p}(0,T;H^{s-2+\frac{\a}{2}}) :\|B\|^p_{W^{\gm,p}(0,T;H^{s-2+\frac{\a}{2}})} \leq R^p  \Bigr \}, \\
V^2_R &=: V^{2,1}_R +V^{2,2}_R = \Bigl \{ B = B_1 +B_2 : B_1 \in V^{2,1}_R \ \ \text{and} \ \ B_2 \in V^{2,2}_R  \Bigr \}.
\end{align*}
We see from the previous compact embeddings that $V^1_R \cap V^2_R$ is compact in $L^2 (0,T;H^s) \cap C([0,T];H^{1}).$ Applying Markov's inequality and Proposition \ref{prop:uni_energy_estimate} \ref{prop:uni_energy_estimate_1} and \ref{prop:uni_energy_estimate_4}, we deduce that 
\begin{align}
\mu^n_B \left( (V^1_R)^c \right) &= \P \Bigl ( \|B_n\|^2_{L^2(0,T;H^{s+\frac{\a}{2}})} + \|B_n \|^2_{W^{\frac{1}{4},2}(0,T;H^{s-2+\frac{\a}{2}})} > R^2 \Bigr ) \notag \\
&\leq \P \Bigl( \|B_n \|^2_{L^2(0,T;H^{s+\frac{\a}{2}})} > \haf R^2 \Bigr) + \P \Bigl( \|B_n \|^2_{W^{\frac{1}{4},2}(0,T;H^{s-2+\frac{\a}{2}})} > \haf R^2 \Bigr) \notag \\
&\leq \frac{2}{R^2} \Bigl( \E  \|B_n\|^2_{L^2(0,T;H^{s+\frac{\a}    {2}})} + \E  \|B_n\|^2_{W^{\frac{1}{4},2}(0,T;H^{s-2+\frac{\a}{2}})} \Bigr) \notag \\
&\leq \frac{C}{R^2} \bigl( 1 + \E \Hsnorm{B_0}^4 \bigr). \label{proof_prop_martingale_solutions_1}
\end{align}
In view of 
\[
\Bigl \{ B_n(t) - \int_0^t \sumk \Pp_n \Tt_k B_n dW^k_\tau \in V^{2,1}_R \Bigr \} \bigcap \Bigl \{ \int_0^t \sumk \Pp_n \Tt_k B_n dW^k_\tau \in V^{2,2}_R \Bigr \} \subset \left \{ B_n(t) \in V^2_R \right \},
\]
we obtain, by Proposition \ref{prop:uni_energy_estimate} \ref{prop:uni_energy_estimate_2} and \ref{prop:uni_energy_estimate_3}, that
\begin{align}
\mu^n_B \left( (V^2_R)^c \right) &\leq \P \Bigl( \lVert B_n - \int_0^\cdot \sumk \Pp_n \Tt_k B_n dW^k_\tau \rVert^2_{W^{1,2} (0,T;H^{s-2+\frac{\a}{2}})} \geq R^2 \Bigr) \notag \\
&+ \P \Bigl( \Big\lVert \int_0^\cdot \sumk \Pp_n \Tt_k B_n dW^{k}_\tau \Big\rVert^p_{W^{\gm,p} (0,T;H^{s-2+\frac{\a}{2}})} \geq R^p \Bigr) \notag \\
&\leq \frac{1}{R^2} \E \Big \lVert B_n - \int_0^\cdot \sumk \Pp_n \Tt_k B_n dW^k_\tau \Big \rVert^2_{W^{1,2} (0,T;H^{s-2+\frac{\a}{2}})} 
+ \frac{1}{R^p}\E \Big \lVert \int_0^\cdot \sumk \Pp_n \Tt_k B_n dW^{k}_\tau \Big\rVert^p_{W^{\gm,p} (0,T;H^{s-2+\frac{\a}{2}})} \notag \\
&\leq \frac{C}{R^2} \bigl( 1 +\E \Hsnorm{B_0}^p \bigr). \label{proof_prop_martingale_solutions_2}
\end{align}
The bounds \eqref{proof_prop_martingale_solutions_1} and \eqref{proof_prop_martingale_solutions_2} imply that 
\[
\mu^n_B \left( (V^1_R \cap V^2_R)^c \right ) \leq \mu^n_B \left ( (V^1_R)^c \right) + \mu^n_B \left ( (V^2_R)^c \right) \leq \frac{C_0}{R^2},
\]
for some $C_0 >0$ independent on $n$. Hence given any $\e > 0$, let $R = \sqrt {\frac{3C_0}{\e}}$ and $K^2_\e := V^1_R \cap V^2_R$ then we have $\mu^n_B \left( K^2_\e \right ) \geq 1 - \frac {\e}{3}.$ Moreover, since it follows directly that $\mu^n_0$ and $\mu^n_\W := \mu_\W$ are both relatively compact, from Theorem \ref{Thm:Prokhorov} we conclude that they are both tight. Hence we may find $K^1_\e \subset H^s$ and $K^3_\e \subset C([0,T];\mathscr{U})$ compact such that 
\[
\mu^n_0 (K^1_\e) \geq 1- \frac{\e}{3} \ \ \text{and} \ \  \mu^n_\W (K^3_\e) \geq 1- \frac{\e}{3}.
\]
Therefore for any $\e >0$, let $K_\e := K^1_\e \times K^2_\e \times K^3_\e$ a compact set in $U$, it holds that $\mu^n(K_\e) \geq 1- \e,$ for all $n \in \N$. As a result, $\{\mu^n \}_{n \in \N}$ is tight in $U$. Again by Theorem \ref{Thm:Prokhorov}, the sequence $\{\mu^n \}_{n \in \N}$ is relatively compact. Therefore we have a weakly convergent subsequence $\mu^{n_j} \to \mu$. Now from Theorem \ref{Thm:Skorokhod} there exists a probability space $\bigl (\widetilde{\Om}, \tFf, \tP \bigr)$, a subsequence of $U$-valued random variables $\bigl(\tB^0_{n_k}, \tB_{n_k}, \widetilde{\W}_{n_k} \bigr)$ such that
\begin{equation}\label{proof_prop_martingale_solutions_3}
\lim_{k \to \infty} \bigl(\tB^0_{n_k}, \tB_{n_k}, \widetilde{\W}_{n_k} \bigr) = \bigl(\tB^0, \tB, \widetilde{\W} \bigr) \ \ \text{in} \ \ U, \ \ \tP\text{-a.s.}
\end{equation}
where the random variable $\bigl(\tB^0, \tB, \widetilde{\W} \bigr)$ has the law $\mu$. In addition, each $\bigl(\tB^0_{n_k}, \tB_{n_k}, \widetilde{\W}_{n_k} \bigr)$ is a martingale solution to \eqref{eq:EMHD_Ito_Galerkin} with $n = n_k$ and the law of $\tB_0$ is the same as that of $B_0$. 
\vspace{1em}\\
\noindent \textbf{Step 2: Passage to the limit}.
We next show that the limit in \eqref{proof_prop_martingale_solutions_3} is indeed the martingale solution to \eqref{eq:EMHD_Ito_Galerkin} with the stochastic basis $\widetilde{S} = \bigl ( \widetilde{\Om}, \tFf , \tF, \tP \bigr)$, where $\tF$ is the filtration generated by $\tB$ and $\tW$. To that end, we will first show the improvement of the regularity of the limit $\tB$ via Proposition \ref{prop:uni_energy_estimate}. Using the Banach-Alaoglu theorem, Proposition \ref{prop:uni_energy_estimate} \ref{prop:uni_energy_estimate_1} with $p=2$, we obtain a further subsequence, still denoted as $\tB_{n_k}$, such that 
\begin{equation}\label{proof_prop_martingale_solutions_4}
\tB_{n_k} \rightharpoonup \tB_1 \ \ \ \text{in} \ \ L^2\bigl(\widetilde{\Om};L^2(0,T;H^{s+\frac{\a}{2}})\bigr),\quad \text{and}
\end{equation}
\begin{equation}\label{proof_prop_martingale_solutions_5}
\tB_{n_k} \overset{\ast}{\rightharpoonup} \tB_2 \ \ \ \text{in} \ \ L^2\bigl(\widetilde{\Om};L^\infty(0,T;H^{s})\bigr),
\end{equation}\\
for some $\tB_1 \in L^2\bigl(\widetilde{\Om};L^2(0,T;H^{s+1})\bigr)$ and $\tB_2 \in L^2\bigl(\widetilde{\Om};L^\infty(0,T;H^{s})\bigr)$. In addition, again by Proposition \ref{prop:uni_energy_estimate} \ref{prop:uni_energy_estimate_1} setting $p=2$, the almost sure convergence \eqref{proof_prop_martingale_solutions_3} and Vitali convergence theorem we have that
\begin{equation}\label{proof_prop_martingale_solutions_6}
\tB_{n_k} \to \tB \ \ \ \text{in} \ \ L^2\bigl(\widetilde{\Om};L^2(0,T;H^s)\bigr),
\end{equation}
hence $\tB = \tB_1$. Furthermore, fix any $\psi \in H^{-s}$ and $M \subset \tOm \times [0,T]$ measurable, then 
\begin{align*}
\E \int_0^T \mathbbm{1}_M \<\tB_{2}, \psi \> dt = \lim_{k \to \infty} \E \int_0^T \mathbbm{1}_M \<\tB_{n_k}, \psi \> dt = \E \int_0^T \mathbbm{1}_M \<\tB, \psi \> dt.
\end{align*}
Thus $\tB = \tB_1 = \tB_2$ and 
\begin{equation}\label{proof_prop_martingale_solutions_7}
\tB \in L^2 \bigl(\tOm; L^2 (0,T;H^{s+1})\bigr) \cap L^2\bigl(\tOm;L^\infty(0,T;H^s)\bigr).
\end{equation}
Next, we prove that $\bigl(\tB^0,\tB,\tW\bigr)$ is a martingale solution by the passage of the limit. Recall that for each $k$ the random variable $\bigl(\tB^0_{n_k},\tB_{n_k},\tW_{n_k}\bigr)$ is a martingale solution to \eqref{eq:EMHD_Ito_Galerkin}, i.e., for any function $\phi \in H$, one has
\begin{align*}
&\<\tB_{n_k}(t), \phi \> + \int_0^t \left < \chi^2_r \Pp_{n_k} \curl (\curl \tB_{n_k}) \times \tB_{n_k} + \lb^{\a} \tB_{n_k}, \phi \right > d\tau \\
= \ &\< \tB^0_{n_k}, \phi \> + \haf \sum_{j=1}^\infty \int_0^t \< \Pp_{n_k} \Tt^2_j \tB_{n_k}, \phi \> d\tau + \sumj \int_0^t \Ltwoinner{\Pp_{n_k}\Tt_j \tB_{n_k},\phi} d\widetilde{W}^{n_k, j}_{\tau} ,
\end{align*}
for all $t \in [0,T]$, almost surely in $\tP$. Recall that the almost sure convergence $\eqref{proof_prop_martingale_solutions_3}$, we first have that 
\begin{equation}\label{proof_prop_martingale_solutions_8}
\Ltwoinner{\tB_{n_k}(t),\phi} \to \Ltwoinner{\tB(t),\phi} \ \ \text{and} \ \ \Ltwoinner{\tB^0_{n_k},\phi} \to \Ltwoinner{\tB^0,\phi}.
\end{equation}
We then proceed to establish the convergence to other linear terms, including: 
\begin{align}
\biggl | \int_0^t \Ltwoinner{\lb^\a \tB_{n_k}, \phi} d\tau - \int_0^t \Ltwoinner{\lb^\a \tB, \phi} d\tau  \biggr | &\les \int_0^T \twonorm{\lb^\a (\tB_{n_k}-\tB)} \twonorm{\phi} dt \notag \\ 
&\les \twonorm{\phi} \biggl( \int_0^T  \Hsnorm{\tB_{n_k}-\tB}^2 dt \biggr)^\haf \longrightarrow 0 ,\label{proof_prop_martingale_solutions_9}
\end{align}
and 
\begin{align}
&\biggl | \haf \sumj \int_0^t \Ltwoinner{\Pp_{n_k} \Tt^2_j \tB_{n_k} ,\phi} d\tau -\haf \sumj \int_0^t \Ltwoinner{\Tt^2_j \tB, \phi} d\tau  \biggr|  \notag \\ 
\leq \ & \haf \sumj \int_0^t \Bigl |\Ltwoinner{\Pp_{n_k} \Tt^2_j \bigl( \tB_{n_k} - \tB \bigr),\Pp_{n_k} \phi} \Bigr | d\tau + \haf \sumj \int_0^t \left| \Ltwoinner{\Tt^2_j \tB, \Pp_{n_k}\phi - \phi} \right| d\tau    \notag \\
\les \ & \ltwoHssnorm{c}^2  \twonorm{\phi} \biggl ( \int_0^T \Hsnorm{\tB_{n_k} - \tB}^2 dt \biggr)^\haf \notag \\
+\ &\ltwoHssnorm{c}^2 \twonorm{\Pp_{n_k}\phi - \phi} \biggl( \int_0^T \Hsnorm{\tB}^2 dt \biggr)^\haf   \longrightarrow 0, \label{proof_prop_martingale_solutions_10}
\end{align}
by the almost sure convergence \eqref{proof_prop_martingale_solutions_3}. As for the convergence of the nonlinearities, we approach it as follows:
\begin{align*}
    &\biggl | \int_0^t \Ltwoinner{\chi^2_r (\tB_{n_k}) \Pp_{n_k} \curl \bigl( (\curl \tB_{n_k}) \times \tB_{n_k}\bigr) , \phi } d\tau - \int_0^t \Ltwoinner{\chi^2_r (\tB) \curl \bigl( (\curl \tB) \times \tB \bigr) , \phi } d\tau  \biggr| \\ 
\leq \   & \int_0^t \left| \Ltwoinner{\chi^2_r (\tB_{n_k}) \curl \bigl( (\curl \tB_{n_k}) \times \tB_{n_k}\bigr) , \Pp_{n_k}\phi -\phi} \right| d\tau \\
+ \ & \int_0^t \left |\Ltwoinner{ \bigl( \chi^2_r (\tB_{n_k}) -\chi_r^2 (\tB) \bigr) \curl \bigl( (\curl \tB_{n_k}) \times \tB_{n_k} \bigr) , \phi } \right| d\tau   \\
+ \ & \int_0^t \left | \Ltwoinner{  \chi^2_r (\tB) \Bigl ( \curl \bigl( (\curl \tB_{n_k}) \times \tB_{n_k} \bigr) -  \curl \bigl( (\curl \tB) \times \tB \bigr) \Bigr)  , \phi  }  \right| d\tau  =: J_1 + J_2 + J_3.
\end{align*}
For $J_1$, we use Sobolev embedding to obtain
\begin{equation}\label{proof_prop_martingale_solutions_11}
J_1 \les \twonorm{\Pp_{n_k}\phi - \phi} \int_0^T \Hsnorm{\tB_{n_k}}^2 dt \to 0, \ \ \ \text{as} \ \ k \to \infty. 
\end{equation}
Similarly, using Sobolev embedding and the fact that $s > \frac{5}{2}$ yields
\begin{align}
J_2 &= \int_0^t \left |\Ltwoinner{ \bigl( \chi^2_r (\tB_{n_k}) -\chi_r^2 (\tB) \bigr) \curl \bigl( (\curl \tB_{n_k}) \times \tB_{n_k} \bigr) , \phi } \right| d\tau \notag\\
&\les \int_0^T \left |\Ltwoinner{|\tB_{n_k} -\tB| \curl (\curl \tB_{n_k}) \times \tB_{n_k}, \phi} \right| dt \notag \\
&\les \biggl( \int_0^T \Hsnorm{\tB_{n_k} - \tB}^2 dt \biggr)^\haf \biggl( \int_0^T \left | \Ltwoinner{\curl (\curl \tB_{n_k}) \times \tB_{n_k},\phi} \right|^2 dt \biggr)^\haf \notag \\
&\les \biggl( \int_0^T \Hsnorm{\tB_{n_k} - \tB}^2 dt \biggr)^\haf \biggl( \int_0^T \Hsnorm{\tB_{n_k}}^2 \sixnorm{\tB_{n_k}}^2 dt \biggr)^\haf \twonorm{\phi} \notag \\
&+ \biggl( \int_0^T \Hsnorm{\tB_{n_k} - \tB}^2 dt \biggr)^\haf \biggl( \int_0^T \Hsnorm{\tB_{n_k}}^2 \|\tB_{n_k}\|^2_{H^{1}} dt \biggr)^\haf \twonorm{\phi} \notag\\
&\les \twonorm{\phi} \sup_{t \in [0,T]} \|\tB_{n_k}\|_{H^{1}} \biggl( \int_0^T \Hsnorm{\tB_{n_k}}^2 dt \biggr)^\haf \biggl( \int_0^T \Hsnorm{\tB_{n_k} - \tB}^2 dt \biggr)^\haf \notag \\ &\les \biggl( \int_0^T \Hsnorm{\tB_{n_k} - \tB}^2 dt \biggr)^\haf \to 0  \ \ \text{as} \ \ k \to \infty, \label{proof_prop_martingale_solutions_12}
\end{align}
where we also utilize the almost sure convergence \eqref{proof_prop_martingale_solutions_3} in $U$. Furthermore, we have 
\begin{align}
J_3 &= \int_0^t \left | \Ltwoinner{  \chi^2_r (\tB) \Bigl ( \curl \bigl( (\curl \tB_{n_k}) \times \tB_{n_k} \bigr) -  \curl \bigl( (\curl \tB) \times \tB \bigr) \Bigr)  , \phi  }  \right| d\tau \notag \\
&\leq \int_0^t \left | \Ltwoinner{\curl \bigl( (\curl \tB_{n_k}) \bigr) \times (\tB_{n_k} - \tB) , \phi} \right | d\tau  + \int_0^t \left | \Ltwoinner{\curl \bigl(\curl \tB_{n_k} - \curl \tB \bigr)  \times  \tB   , \phi} \right | d\tau \notag \notag \\
&\les \int_0^T \Hsnorm{\tB_{n_k}} \Linfnorm{\tB_{n_k} - \tB} \twonorm{\phi} dt \notag \\ 
&+ \twonorm{\phi} \int_0^T \left ( \Linfnorm{\tB} \Hsnorm{\tB_{n_k} - \tB} + \Linfnorm{\nabla \tB} \Honenorm{\tB_{n_k} - \tB}  \right )dt \notag \\
&\les \biggl ( \Bigl( \int_0^T \Hsnorm{\tB_{n_k}}^2 dt \Bigr)^\haf +  \Bigl( \int_0^T \Hsnorm{\tB}^2 dt \Bigr)^\haf  \biggr ) \biggl( \int_0^T \Hsnorm{\tB_{n_k} - \tB}^2 dt \biggr)^\haf \to 0. \label{proof_prop_martingale_solutions_12}
\end{align}\\
Finally, we show the convergence of the stochastic terms. Note that
\begin{align}
    &\sup_{t \in [0,T]} \biggl|  \sumj \Ltwoinner{\Pp_{n_k} \Tt_j \tB_{n_k},\phi} - \sumj \Ltwoinner{\Tt_j \tB,\phi}  \biggr| \notag \\
 \les \  &\sup_{t \in [0,T]}  \biggl(  \sumj \left|\Ltwoinner{ \Tt_j \tB_{n_k},\Pp_{n_k}\phi - \phi} \right| + \sumj \left| \Ltwoinner{\Tt_j \bigl ( \tB_{n_k} -\tB \bigr),\phi}  \right| \biggr) \notag \\
  \les \  & \ltwoHssnorm{c} \biggl( \twonorm{\Pp_{n_k} \phi - \phi} \sup_{t\in [0,T]} \Honenorm{\tB_{n_k}} + \twonorm{\phi} \Honenorm{\tB_{n_k} - \tB}  \biggr) \to 0,   \label{proof_prop_martingale_solutions_13}
\end{align}
almost surely as $k \to \infty$. We again used the almost sure convergence of $\tB_{n_k} \to \tB$ in $C([0,T];H^{1})$ from \eqref{proof_prop_martingale_solutions_3}. In addition, since we also have $\tW_{n_k} \to \tW$ almost surely in $C([0,T];\mathscr{U})$. By Lemma~\ref{Lemma:stochastic_convergence} we get 
\begin{equation}\label{proof_prop_martingale_solutions_14}
\sumj \int_0^t \Ltwoinner{\Pp_{n_k} \Tt_j \tB_{n_k},\phi} d\tw^{n_k,j}_{\tau} \longrightarrow \sumj \int_0^t \Ltwoinner{\Tt_j \tB,\phi} d\tw^j_{\tau},
\end{equation}
in $C([0,T];\R)$ in probability. We can then pass a further subsequence of $\tW_{n_k}$ such that \eqref{proof_prop_martingale_solutions_14} is convergent almost surely in $C([0,T];\R)$. Therefore we have shown that $\tB$ satisfies Definition~\ref{def_Martingale_solution} of the equation
\begin{equation}\label{proof_prop_martingale_solutions_15}
d\tB + \Bigl ( \chi^2_r \curl \bigl( ( \curl \tB )\times \tB \bigr) +\mu \lb^\a \tB \Bigr) dt = \haf \sum_{k =1}^\infty \Pp_n \Tt^2_k \tB dt + \sum_{k=1}^\infty \Pp_n \Tt_k \tB d\tw^k.
\end{equation}
Now it suffices to show the time regularity for $\tB$, i.e.,
$\tB \in L^2 (\tOm; C([0,T],H^s)).$ Using density argument, one can first show that for each sample path, $t \longrightarrow \Ltwoinner{\tB(t),\phi} $ is weakly continuous in $H^{s}$, $\tP$-a.s. Hence we only need to show that the map $t \longrightarrow \Hsnorm{\tB}$ is continuous $\tP$-a.s. \\

Let $\e > 0$, define the standard spatial mollifier $\phi_\e$ and the mollification operator $J_\e$ as 
\[
J_\e f(x) := \left( \phi_\e \ast f \right) (x).
\]
See also Appendix \ref{appendix:proof_lemma_stochastic_convergence} for a similar definition on the temporal mollifier. Now applying the operator to equation \eqref{proof_prop_martingale_solutions_15} brings 
\begin{align*}
J_\e \tB (t) + \int_0^t J_\e \Bigl ( \chi^2_r \curl \bigl ( ( \curl \tB )\times \tB \bigr )  +\mu \lb^\a \tB \Bigr) d\tau  = J_\e \tB(0) + \haf \sumk \int_0^t  J_\e \Tt^2_k \tB d\tau + \sumk \int_0^t J_\e \Tt_k \tB d\tw^k_\tau. 
\end{align*}
Thanks to It\^o's formula to $\twonorm{\lb^s J_\e \tB(t)}^2$ as in Lemma~\ref{lemma:Ito_formula_EMHD}, we arrive at
\begin{align*}
    \twonorm{\lb^s J_\e \tB(t) }^2 &= \twonorm{\lb^s J_\e \tB(0)}^2 - 2 \int_0^t \Ltwoinner{J_\e \chi_r^2 \bigl(\curl (\curl \tB) \times \tB \bigr) + J_\e \lb^\a \tB, \lb^{2s} J_\e \tB } d\tau \\ &+ 2  \sumk \int_0^t \left( \Ltwoinner{J_\e \Tt^2_k \tB, \lb^{2s} J_\e \tB } + \twonorm{\lb^s J_\e \Tt_k \tB}^2 \right) d\tau \\
    &+ 2  \sumk \int_0^t \Ltwoinner{J_\e \Tt_k \tB, J_\e \lb^{2s} \tB } d\tw^k_\tau. 
\end{align*}
Each time integrals of the above equation can be estimated in a similar manner as \eqref{proof:uniform_energy_estimate_2}, \eqref{proof:uniform_energy_estimate_4} and \eqref{proof:uniform_energy_estimate_6}, hence we have 
\begin{align*}
\int_0^t \left| \Ltwoinner{J_\e \chi_r^2 \bigl(\curl (\curl \tB) \times \tB \bigr) + J_\e \lb^\a \tB, \lb^{2s} J_\e \tB } \right| d\tau &\les \int_0^t \Hsnorm{\tB}^2 d\tau,  \\
\sumk \int_0^t \left( \Ltwoinner{J_\e \Tt^2_k \tB, \lb^{2s} J_\e \tB } + \twonorm{\lb^s J_\e \Tt_k \tB}^2 \right) d\tau &\les \ltwoHssnorm{c_\ell}^2 \int_0^t \Hsnorm{\tB}^2 d\tau. 
\end{align*}
For the stochastic term, using Burkholder-Davis-Gundy Inequality (Theorem~\ref{Thm:Burkholder-Davis-Gundy Inequality}) and Lemma~\ref{lemma:commutator_est_1} yields
\begin{align*}
&\E \sup_{t \in [0,T]} 2 \Bigl| \int_0^t \sumk \left( \Ltwoinner{J_\e \Tt_k \tB, J_\e \lb^{2s} \tB } - \Ltwoinner{\Tt_k \tB,\lb^{2s}\tB} \right) d\tw^k_\tau \Bigr| \\
=\ & 2 \E \sup_{t \in [0,T]} \Bigl| \int_0^t \sumk \Ltwoinner{\lb^s \Tt_k \tB, \lb^s J^2_\e \tB - \lb^s \tB} d\tw^k_\tau  \Bigr| \\
\les \ &\E \biggl( \int_0^T \sumk \Ltwoinner{\lb^s \Tt_k \tB, \lb^s J_\e^2 \tB -\lb^s \tB}^2 d\tau \biggr)^\haf  \\
\end{align*}
\begin{align*}
\les \ & \E \biggl( \int_0^T \sumk \Ltwoinner{ [\lb^s, \Tt_k ] \tB, \lb^s J_\e^2 \tB -\lb^s \tB}^2 d\tau \biggr)^\haf + \E \biggl( \int_0^T \sumk \Ltwoinner{ [J_\e, \Tt_k ] \lb^s \tB, \lb^s J_\e \tB}^2 d\tau \biggr)^\haf \\
\les \ & \E \biggl( \int_0^T \Bigl( \sumk \twonorm{[\lb^s, \Tt_k ] \tB}^2\Bigr) \twonorm{\lb^s J_\e^2 \tB -\lb^s \tB}^2 d\tau \biggr)^\haf \\
+ \ & \E \biggl( \int_0^T \Bigl( \sumk  \twonorm{ [J_\e, \Tt_k ] \lb^s \tB}^2 \Bigr) \twonorm{ \lb^s J_\e \tB}^2 d\tau \biggr)^\haf \\
\les \ & \E \biggl( \sup_{t \in [0,T]} \Hsnorm{\tB(t)} \Bigl( \int_0^T \twonorm{\lb^s J_\e^2 \tB -\lb^s \tB}^2 d\tau \Bigr)^\haf  \biggr) \\
+ \ & \E \biggl( \sup_{t \in [0,T]} \Hsnorm{\tB(t)} \Bigl( \int_0^T \sumk \twonorm{ [J_\e,\Tt_k] \tB}^2 d\tau \Bigr)^\haf  \biggr) \\
\les \ & \d \E \Bigl( \sup_{t \in [0,T]} \Hsnorm{\tB} \Bigr) + \frac{1}{\d} \E \biggl( \int_0^T \twonorm{\lb^s J_\e^2 \tB -\lb^s \tB}^2 d\tau + \int_0^T \sumk \twonorm{ [J_\e,\Tt_k] \tB}^2 d\tau \biggr) \to 0,
\end{align*}
by first sending $\e \to 0$ and then $\d \to 0$. Here we use the property of the mollification operator $J_\e$. Thus we can extract a sequence $\e_n \to 0$ such that 
\[
\lim_{n \to \infty } 2 \int_0^t \sumk  \Ltwoinner{J_{\e_n} \Tt_k \tB, J_{\e_n} \lb^{2s} \tB } d\tw^k_\tau =  2 \int_0^t \sumk \Ltwoinner{\Tt_k \tB,\lb^{2s}\tB}  d\tw^k_\tau, 
\]
$\tP$-a.s in $C([0,T];\R)$. As a result, for any $t,s \in [0,T]$ and $\tP$-a.s, it holds that
\begin{align*}
&\left| \Hsnorm{\tB(t)}^2 -\Hsnorm{\tB(s)}^2 \right| = \lim_{n \to \infty} \left| \Hsnorm{J_{\e_n} \tB(t)}^2 -\Hsnorm{J_{\e_n} \tB(s)}^2 \right| \\
\les \ &  \int_s^t  \Hsnorm{\tB(\tau)}^2   d\tau +\Bigl| \int_s^t \sumk \Ltwoinner{\Tt_k \tB,\lb^{2s}\tB}  d\tw^k_\tau \Bigr|,
\end{align*}
which implies the continuity of the map $t \to \Hsnorm{\tB}$. Combining with the weak continuity, we conclude that $\tB \in L^2 (\tOm;C([0,T];H^s))$. 
\end{proof}

\section{Uniqueness and proof of the main result}\label{Sec:proof_of_main_Thm}
In this section, we first establish pathwise uniqueness and then combine it with the existence result and the Yamada--Watanabe--type theorem to complete the proof of the main result.
\subsection{Pathwise uniqueness} We establish in this section the pathwise uniqueness of martingale solutions of the system \eqref{eq:EMHD_cutoff_Ito_1}-\eqref{eq:EMHD_cutoff_Ito_3} obtained from Proposition \ref{prop:martingale_solution}. 

\begin{Proposition}\label{prop:pathwise_uniqueness}
Under the same assumptions as in Proposition \ref{prop:uni_energy_estimate}, let $\left(\Ss,\W,B_1 \right)$ and $\left(\Ss,\W,B_2 \right)$ be two martingale solution on the cutoff system \eqref{eq:EMHD_cutoff_Ito_1} - \eqref{eq:EMHD_cutoff_Ito_3} on $[0,T]$ with the same initial data $B_0 \in L^2(\Om;H^s)$, the same stochastic basis $\Ss$ as well as noise $\W$. Suppose further that $s > 3$, then the following holds:
\[
\P\left( B_1 = B_2, \forall t \in [0,T] \right) = 1.
\]
\end{Proposition}
\begin{proof} Denote $\barB(t) = B_1 (t) - B_2 (t)$, then $\barB (t)$ is a martingale solution to 
\begin{align*}
d\barB + \Bigl ( \chi_r^2 (\Woneinfnorm{B_1} ) \curl \bigl ((\curl B_1) \times B_1 \bigr )   &- \chi_r^2 (\Woneinfnorm{B_2} ) \curl \bigl ((\curl B_2) \times B_2 \bigr ) + \lb^\a  \barB \Bigr ) dt \\
&=  \haf \sumk \Tt_k^2 \barB dt + \sumk \Tt_k \barB dW^k,
\end{align*}
with initial data $\barB(0) = 0$. It\^o formula and integration by part imply that 
\begin{align*}
d \twonorm{\barB}^2 &= 2 \Ltwoinner{\chi_r^2 (\Woneinfnorm{B_1} ) \curl \bigl ((\curl B_1) \times B_1 \bigr) , \barB} dt \\
&- 2 \Ltwoinner{\chi_r^2 (\Woneinfnorm{B_2} ) \curl \bigl ((\curl B_2) \times B_2 \bigr) , \barB} dt + 2\Ltwoinner{\lb^\a \barB, \barB} dt \\
&+ \sumk \Bigl( \Ltwoinner{\Tt_k^2 \barB, \barB} + \twonorm{\Tt_k \barB}^2 \Bigr) dt + 2 \sumk \Ltwoinner{\Tt_k\barB, \barB} dW^k. \\
\Rightarrow d \twonorm{\barB}^2&= 2 \Ltwoinner{\left( \chi_r^2 (\Woneinfnorm{B_1}) -\chi_r^2 (\Woneinfnorm{B_2}) \right) \curl \bigl ((\curl B_1) \times B_1 \bigr ) , \barB} dt \\
&+ 2 \Ltwoinner{\chi_r^2 (\Woneinfnorm{B_2} ) \bigl( \curl  ((\curl B_1) \times B_1  ) - \curl  ( (\curl B_2) \times B_2 ) ) \bigr) , \barB} dt \\
&+ 2\Ltwoinner{\lb^\a \barB, \barB} dt.
\end{align*}
Note that
\begin{align}
&\curl \left( (\curl B_1) \times B_1 \right) - \curl \left( (\curl B_2) \times B_2 \right)  \notag \\
=\ & \curl \left( (\curl B_1) \times \barB \right) + \curl \left( (\curl B_1) \times B_2 \right)  - \curl \left( (\curl B_2) \times B_2 \right)  \notag \\
=\ & \curl \left(  (\curl B_1) \times \barB \right) + \curl \left( (\curl \barB ) \times B_2 \right).          \label{proof:prop_pathwise_uniqueness_1}
\end{align}
Using integration by part we have
\begin{align*}
&2 \Ltwoinner{\chi_r^2 (\Woneinfnorm{B_2} ) \bigl( \curl ((\curl B_1) \times B_1 ) - \curl ((\curl B_2) \times B_2 ) \bigr) , \barB} dt \\
=\ & 2 \Ltwoinner{\chi_r^2 (\Woneinfnorm{B_2} ) \bigl( \curl  ((\curl B_1) \times \barB ) + \curl ((\curl \barB ) \times B_2 ) \bigr) , \barB} dt \\
=\ & 2 \Ltwoinner{\chi_r^2 (\Woneinfnorm{B_2} ) \curl \bigl ((\curl B_1) \times \barB \bigr) , \barB} dt,\\
=\ & 2 \Ltwoinner{\chi_r^2 (\Woneinfnorm{B_2} ) \bigl( \curl  ((\curl \barB) \times \barB ) + \curl ((\curl B_2) \times \barB ) \bigr) , \barB} dt  \\
=\ & 2 \Ltwoinner{\chi_r^2 (\Woneinfnorm{B_2} ) \curl \bigl ((\curl B_2) \times \barB \bigr) , \barB} dt,
\end{align*}
which results in
\begin{align*}
d \twonorm{\barB}^2 &+ 2 \Ltwoinner{\left( \chi_r^2 (\Woneinfnorm{B_1}) -\chi_r^2 (\Woneinfnorm{B_2}) \right) \curl \bigl ((\curl B_1) \times B_1 \bigr) , \barB} dt \\
&+2 \Ltwoinner{\chi_r^2 (\Woneinfnorm{B_2})\curl \bigl ((\curl B_1) \times \barB \bigr ),\barB} dt + \twonorm{\lb^{\frac{\a}{2}} \barB }^2 dt  = 0.
\end{align*}
We then proceed to estimate the nonlinear terms of the above as the following: 
\begin{align*}
     2 \Ltwoinner{\left( \chi_r^2 (\Woneinfnorm{B_1}) -\chi_r^2 (\Woneinfnorm{B_2}) \right) \curl \bigl ((\curl B_1) \times B_1 \bigr) , \barB} &\les \Woneinfnorm{\barB}^2 \Honenorm{B_1}\twonorm{B_1} \\
     &\les \HsmhNorm{\barB}^2 \Hsnorm{B_1}^2, \\ 
\text{and} \quad 2 \Ltwoinner{\chi_r^2 (\Woneinfnorm{B_2})\curl \bigl ((\curl B_1) \times \barB \bigr ),\barB}  &\leq \Woneinfnorm{\barB} \Honenorm{B_1} \twonorm{\barB}  \\
&\les \HsmhNorm{\barB}^2 \Hsnorm{B_1},
\end{align*}
where we used the fact that $s-\haf > \frac{5}{2}$ together with the Sobolev embedding $H^{s-\haf} \hookrightarrow W^{1,\infty}$. We therefore deduce that 
\begin{equation}\label{proof:prop_pathwise_uniqueness_2}
d \twonorm{\barB}^2 + 2 \twonorm{\lb^{\frac{\a}{2}} \barB}^2 dt \les  \HsmhNorm{\barB}^2 \bigl( 1+ \Hsnorm{B_1}^2 \bigr ) dt.
\end{equation}
Now we estimate the $H^{s-\haf}$-norm of $\barB$. By virtue of It\^o's formula and integration by part, we have 
\begin{align*}
d \twonorm{\lb^{s-\haf} \barB }^2 &+ 2 \Ltwoinner{\lb^\smh  \bigl( \chi^2_r (\Woneinfnorm{B_1}) \curl ((\curl B_1) \times B_1\bigr)    , \lb^\smh \barB } dt \\
&- 2 \Ltwoinner{\lb^\smh \bigl( \chi^2_r (\Woneinfnorm{B_2}) \curl ((\curl B_2) \times B_2)\bigr),\lb^\smh \barB } dt + 2 \mu \Ltwoinner{\lb^\a \barB, \lb^{2s-1} \barB} dt \\
&= \sumk \Bigl( \Ltwoinner{\lb^\smh \Tt_k^2 \barB , \lb^\smh \barB} + \twonorm{\lb^\smh \Tt_k \barB}^2 \Bigr) dt + 2 \sumk \Ltwoinner{\lb^\smh \Tt_k \barB, \lb^s \barB} dW^k.
\end{align*}
The linear term is handled similarly as in \eqref{proof:uniform_energy_estimate_1}, 
\begin{align*}
\sumk \Bigl( \Ltwoinner{\lb^\smh \Tt_k^2 \barB , \lb^\smh \barB} + \twonorm{\lb^\smh \Tt_k \barB}^2 \Bigr) \les \ltwoHssnorm{c_\ell}^2 \HsmhNorm{\barB}^2.
\end{align*}
As for the nonlinear term, we use the identity 
\[
a^2b-c^2d \equiv  (a-c) (ab+cd) + ac(b-d),
\]
to deduce that 
\begin{align*}
  &2 \Ltwoinner{\lb^\smh  \bigl( \chi^2_r (\Woneinfnorm{B_1}) \curl ((\curl B_1) \times B_1) \bigr)   , \lb^\smh \barB }  \\
-\ &2 \Ltwoinner{\lb^\smh \bigl( \chi^2_r (\Woneinfnorm{B_2}) \curl ((\curl B_2) \times B_2) \bigr),\lb^\smh \barB } \\
= \ & 2 \Ltwoinner{\lb^\smh \bigl( \chi_r (\Woneinfnorm{B_1}) - \chi_r(\Woneinfnorm{B_2})\bigr) \chi_r (\Woneinfnorm{B_1}) \curl ( (\curl B_1) \times B_1 ), \lb^\smh \barB } \\
+ \ & 2 \Ltwoinner{\lb^\smh \bigl( \chi_r (\Woneinfnorm{B_1}) - \chi_r(\Woneinfnorm{B_2})\bigr) \chi_r (\Woneinfnorm{B_2}) \curl ( (\curl B_2) \times B_2 ), \lb^\smh \barB } \\
+\ & 2 \Ltwoinner{\lb^\smh \bigl( \chi_r (\Woneinfnorm{B_1}) \chi_r(\Woneinfnorm{B_2}) \curl ( (\curl B_1) \times B_1 ) - \curl ( (\curl B_2) \times B_2 ) \bigr), \lb^\smh \barB }. 
\end{align*}
For the first two terms of the above, we use the Lipschitz continuity of the cut-off function $\chi_r$, Lemma \ref{lemma:commutator_est_0} and Sobolev embedding to obtain
\begin{align*}
 & 2 \Ltwoinner{\lb^\smh \bigl( \chi_r (\Woneinfnorm{B_1}) - \chi_r(\Woneinfnorm{B_2})\bigr) \chi_r (\Woneinfnorm{B_1}) \curl ( (\curl B_1) \times B_1 ), \lb^\smh \barB } \\
=  \ & 2 \Ltwoinner{\lb^\smh \bigl( \chi_r (\Woneinfnorm{B_1}) - \chi_r(\Woneinfnorm{B_2})\bigr) \chi_r (\Woneinfnorm{B_1})( (\curl B_1) \times B_1 ), \lb^\smh \curl \barB }\\
\les \ & \Woneinfnorm{\barB} \chi_r (\Woneinfnorm{B_1})\Ltwoinner{\lb^\smh ( (\curl B_1) \times B_1 ), \lb^\smh \curl \barB } \\
\les \ & \HsmhNorm{\barB} \HsaNorm{\barB}\chi_r (\Woneinfnorm{B_1}) \bigl( \HsaNorm{B_1} \Linfnorm{B_1} + \Woneinfnorm{B_1}\|B_1\|_{H^{s-\haf}(\T^3)} \bigr)\\
\les \ & r \HsmhNorm{\barB} \HsaNorm{\barB} \HsaNorm{B_1} \\
\leq \ & \frac{\mu}{4}\HsaNorm{\barB}^2 + C_r \HsmhNorm{\barB}^2 \HsaNorm{B_1}^2,
\end{align*}
and similarly as 
\begin{align*}
     & 2 \Ltwoinner{\lb^s \bigl( \chi_r (\Woneinfnorm{B_1}) - \chi_r(\Woneinfnorm{B_2})\bigr) \chi_r (\Woneinfnorm{B_2}) \curl ( (\curl B_2) \times B_2 ), \lb^s \barB } \\
    \leq \ & \frac{\mu}{4}\HsaNorm{\barB}^2 + C_r \HsmhNorm{\barB}^2 \HsaNorm{B_2}^2.
\end{align*}
The last inequality from the above follows from Young's inequality. In view of \eqref{proof:prop_pathwise_uniqueness_1} we can write 
\begin{align*}
   & 2 \Ltwoinner{\lb^\smh \bigl( \chi_r (\Woneinfnorm{B_1}) \chi_r(\Woneinfnorm{B_2}) \curl ((\curl B_1) \times B_1 ) - \curl ( (\curl B_2) \times B_2 ) \bigr), \lb^\smh \barB } \\
 \les \ &2 \Ltwoinner{\lb^\smh \left( \chi_r(\Woneinfnorm{B_1}) \chi_r(\Woneinfnorm{B_2}) \left( \curl \left(  (\curl B_1) \times \barB \right) + \curl \left( (\curl \barB ) \times B_2 \right ) \right)  \right) , \lb^\smh \barB } \\
 \les \ & 2 \Ltwoinner{\lb^\smh \left( \chi_r(\Woneinfnorm{B_1}) \chi_r(\Woneinfnorm{B_2}) \left(   (\curl B_1) \times \barB  +  (\curl \barB ) \times B_2 \right ) \right) , \lb^\smh (\curl \barB) } \\
= \ &2 \Ltwoinner{\lb^\smh \left( \chi_r(\Woneinfnorm{B_1}) \chi_r(\Woneinfnorm{B_2})   (\curl B_1) \times \barB   \right) , \lb^\smh (\curl \barB) }\\
+ \ &2 \Ltwoinner{\lb^\smh \left( \chi_r(\Woneinfnorm{B_1}) \chi_r(\Woneinfnorm{B_2}) (\curl \barB ) \times B_2  \right) , \lb^\smh (\curl \barB) } =: N_1 + N_2,
\end{align*}
where $N_1$ can be bounded by 
\begin{align*}
N_1 &\les \HsaNorm{\barB} \left( \HsaNorm{B_1} \Linfnorm{\barB} + \HsNorm{\barB} \Linfnorm{\curl B_1} \right) \\
&\les r \HsaNorm{\barB} \HsmhNorm{\barB} \left( \HsaNorm{B_1} + \Hsnorm{B_1} \right) \\
&\leq \frac{\mu}{4} \HsaNorm{\barB}^2 + C_r \HsmhNorm{\barB}^2 \HsaNorm{B_1}^2.
\end{align*}
To estimate $N_2$, we use the second part of Lemma~\ref{lemma:commutator_est_0}:
\begin{align*}
N_2 &= 2  \chi_r(\Woneinfnorm{B_1}) \chi_r(\Woneinfnorm{B_2}) \Ltwoinner{\lb^\smh \left( (\curl \barB ) \times B_2  \right) ,  \curl (\lb^\smh \barB) } \\
&-2  \chi_r(\Woneinfnorm{B_1}) \chi_r(\Woneinfnorm{B_2}) \Ltwoinner{ (\curl \lb^\smh \barB ) \times B_2   ,  \curl (\lb^\smh \barB) } \\
&\les  \chi_r(\Woneinfnorm{B_1}) \chi_r(\Woneinfnorm{B_2}) \HsaNorm{\barB} \left( \twonorm{\lb^\smh \barB} \Linfnorm{\nabla B_2} + \Linfnorm{\curl \barB} \Hsnorm{B_2} \right) \\
&\les r \HsaNorm{\barB}  \HsmhNorm{\barB}  \Hsnorm{B_2} \\
&\leq \frac{\mu}{4} \HsaNorm{\barB}^2 + C_r \HsmhNorm{\barB}^2 \Hsnorm{B_2}^2.
\end{align*}
Consequently, we have 
\begin{align}
    &2 \Ltwoinner{\lb^\smh \bigl( \chi_r (\Woneinfnorm{B_1}) \chi_r(\Woneinfnorm{B_2}) \curl ( (\curl B_1) \times B_1 ) - \curl ( (\curl B_2) \times B_2 ) \bigr), \lb^\smh \barB} \notag \\
\leq \ & \frac{\mu}{2} \HsaNorm{\barB}^2 + C_r \HsmhNorm{\barB}^2 \bigl( \HsaNorm{B_1}^2 + \HsaNorm{B_2}^2 \bigr), \label{proof:prop_pathwise_uniqueness_3}
\end{align}
and hence 
\begin{align*}
 &2 \Ltwoinner{\lb^\smh  \bigl( \chi^2_r (\Woneinfnorm{B_1}) \curl ((\curl B_1) \times B_1) \bigr)   , \lb^\smh \barB }  \\
-\ &2 \Ltwoinner{\lb^\smh \bigl( \chi^2_r (\Woneinfnorm{B_2}) \curl ((\curl B_2) \times B_2) \bigr),\lb^\smh \barB } \\
\leq \  &\mu \HsaNorm{\barB}^2 + \HsmhNorm{\barB}^2 \bigl( \HsaNorm{B_1}^2 + \HsaNorm{B_2}^2 \bigr). \\
\end{align*}
We then arrive at 
\begin{align}
d \twonorm{\lb^\smh \barB}^2 + \mu \HsaNorm{\barB}^2 dt &\les   \HsmhNorm{\barB}^2 \bigl(\HsaNorm{B_1}^2 + \HsaNorm{B_2}^2 \bigr) dt \notag \\
&+ 2\sumk \Ltwoinner{\lb^s \Tt_k \barB, \lb^s \barB} dW^k.\notag
\end{align}
Taking into account of equation \eqref{proof:prop_pathwise_uniqueness_2} we obtain
\begin{equation}\label{proof:prop_pathwise_uniqueness_4}
d \HsmhNorm{\barB}^2 \leq  C_r \HsmhNorm{\barB}^2 \bigl(1+ \HsaNorm{B_1}^2 + \HsaNorm{B_2}^2 \bigr) dt + 2\sumk \Ltwoinner{\lb^s \Tt_k \barB, \lb^s \barB} dW^k.
\end{equation}
Denote 
\[
U_t := \exp \Bigl \{-C_r \int_0^t \left( 1+ \HsaNorm{B_1}^2 + \HsaNorm{B_2}^2 \right) d\tau \Bigr \}, 
\]
we have by It\^o's formula, 
\[
d \bigl( U_t \HsmhNorm{\barB}^2 \bigr) \leq 2 U_t  \sumk \Ltwoinner{\lb^s \Tt_k \barB, \lb^s \barB} dW^k, \quad \text{or}
\]
\[
U_t \HsmhNorm{\barB}^2  \leq \HsmhNorm{\barB(0)}^2 + 2  \sumk \int_0^t U_\tau \Ltwoinner{\lb^s \Tt_k \barB, \lb^s \barB} dW^k_\tau. 
\]
The stochastic integral on the right hand side of the above is a martingale as $\barB \in L^2 (\tOm; L^2 (0,T;H^{s+\frac{\a}{2}})) \cap L^2(\tOm;L^\infty(0,T;H^s)) $. Taking expectation on both sides brings 
\[
\E \bigl ( U_t \HsmhNorm{\barB}^2 \bigr ) \leq 0,
\]
for $\barB(0) = 0$. By the virtue of $U_t > 0$, we must have $\HsmhNorm{\barB (t)} = 0$ almost surely for all $t \in [0,T]$. This concludes the proof of the pathwise uniqueness.  
\end{proof}

\begin{Remark}
In fact, a weaker condition $s + \frac{\a}{2} \geq \frac{7}{2}$ is sufficient for Proposition~\ref{prop:pathwise_uniqueness}.
\end{Remark}

\subsection{Proof of Theorem \ref{Thm:main}}
Using Propositions \ref{prop:martingale_solution} and \ref{prop:pathwise_uniqueness} as well as the Yamada-Watanabe--type theorem (for the proof, see, for example, \cite[Chapter E]{prevot2007concise}), we can establish a unique pathwise solution 
\[
B \in L^2\bigl( \Om; C([0,T],H^s(\T^3))\cap L^2((0,T);H^{s+1}) \bigr)
\] 
to the system \eqref{eq:EMHD_cutoff_Ito_1}-\eqref{eq:EMHD_cutoff_Ito_3}, for any $T > 0$. Now let $\s_r$ be the stopping time 
\begin{equation}\label{def:stopping_time}
\s_r := \inf \bigl \{ t \geq 0: \Woneinfnorm{B(t)} > \frac{r}{2} \bigr \},
\end{equation}
where $r > 0$ is the same one appears in the cutoff function $\chi_r$. Denote the constant for Sobolev embedding $H^s \subset W^{1,\infty}$ by $C_1$, then $\left( B, \s_r \right)$ is a local pathwise solution of the system \eqref{eq:EMHD_Ito} for each $r \geq 2C_1 \bigl( \Hsnorm{B_0} + 1 \bigr)$.

Now consider the initial conditions $B_0 \in L^2 (\Om; H^s)$. Firstly, we define $B_0^k$ for each integer $k \in \N$ as
\[
B_0^k := B_0 \mathbbm{1}_{ \{ k-1 \leq \Hsnorm{B_0} \leq k \} },
\]
and hence $B_0^k \in L^\infty (\Om;H^s(\T^3))$. Repeating the above argument leads to a sequence of local pathwise solution $\left( B_k , \s_{r_k} \right)$ with $r_k = 2C_1 (k+1)$. Secondly, fix some $L > 1$, we let
\[
B = \sumk B_k \mathbbm{1}_{\{k-1 \leq \Hsnorm{B_0} \leq k\}}, \ \ \ 
\tau = \sumk \tau_k \mathbbm{1}_{\{k-1 \leq \Hsnorm{B_0} \leq k\}},\quad\text{where}
\]
\begin{equation*}
\tau_k := \s_{r_k} \wedge \inf \Bigl \{ t \geq 0: \sup_{t^\ast \in [0,t \wedge \s_{r_k}]} \Hsnorm{B_k(t^\ast)}^2 + \int_0^{t\wedge \s_{r_k}} \HsaNorm{B_k(t^\ast)}^2 dt^\ast \geq L + \Hsnorm{B_0}^2\Bigr \}.
\end{equation*}
It follows from the definition of $\tau_k$ that $\tau > 0$, $\P$-almost surely and it is a stopping time. We hence achieve a local pathwise solution $(B,\tau)$ to \eqref{eq:EMHD_Ito} for initial condition $B_0 \in L^2(\Om;H^s)$. Indeed, we have 
\begin{align*}
    &\E \Bigl( \sup_{t \in [0,\tau]} \Hsnorm{B(t)}^2 + \int_0^\tau \HsaNorm{B(t)}^2 dt \Bigr) \\
    =\ &\E \sumk \mathbbm{1}_{\Om_k} \Bigl( \sup_{t \in [0,\tau]} \Hsnorm{B_k(t)}^2 + \int_0^\tau \HsaNorm{B_k(t)}^2 dt \Bigr) \leq \E \sumk \mathbbm{1}_{\Om_k} \bigl( L + \Hsnorm{B_0}^2 \bigr) \\
    \leq\ &L + \E \Hsnorm{B_0}^2 < \infty.
\end{align*}

 Next, denote the collection $\Gm$ containing all stopping time corresponding to a local pathwise solution with the initial condition $B_0$. Then there exists a stopping time $\xi$ such that $\xi \geq \s \in \Gm$ (see \cite{Doob94} Chapter V, section 18 ). Furthermore, we can find a sequence of stopping time $\left\{\s_n \right\}_{n \geq 1} \subset \Gm$ such that $\s_n \leq \s_{n+1}$  and
\[
\lim_{n \to \infty } \s_n = \xi,
\]
almost surely. Let $(B_n,\s_n)$ be the local pathwise solutions for each $n \in \N$. Finally, we conclude that the solution $(B,(\s_n)_{n \in \N},\xi)$ given by
\[
B(\om,t,x) = \lim_{n \to \infty} B_n  (\om, t \wedge \s_n) \mathbbm{1}_{[0,\xi)]}(\om)
\]
is the maximal pathwise solution to the original system. The detail of the argument can be found in \cite[Section 3.4]{BS21}. \qed

\section{Conclusion and discussion}\label{Sec:conclusion}
In this paper, we established the local pathwise well-posedness of the three-dimensional stochastic EMHD system driven by multiplicative transport noise on the torus with fractional dissipation. The main analytical challenge lies in the interplay between the derivative-intensive Hall nonlinearity and the stochastic transport operators, particularly in the regime $\alpha<2$, where the dissipation is not sufficiently strong to directly compensate for the loss of derivatives. To address this difficulty, we developed a cutoff approximation framework together with refined high-order Sobolev energy estimates based on Littlewood--Paley analysis and commutator estimates. This approach yields the existence of martingale solutions, pathwise uniqueness, and, via the Yamada--Watanabe--type theorem, the existence of unique maximal pathwise solutions. Several directions remain open for future study, including global well-posedness under additional assumptions, the long-time behavior of solutions, and possible regularizing effects induced by the transport noise. More broadly, the analytical framework developed here may also be applicable to other stochastic fluid models with derivative-loss nonlinearities and transport-type noise.

\section*{Acknowledgments} 
This work was partially supported by the ONR grant under \#N00014-24-1-2432, the
Simons Foundation (MP-TSM-00002783) and the NSF grant DMS-2420988.

\bibliographystyle{plain}
\bibliography{EMHD_transport_noise}

@book{da2014stochastic,
  title={Stochastic equations in infinite dimensions},
  author={Da Prato, Giuseppe and Zabczyk, Jerzy},
  volume={152},
  year={2014},
  publisher={Cambridge university press}
}

@article {HLL26,
    AUTHOR = {Hu, Ruimeng and Lin, Quyuan and Liu, Rongchang},
     TITLE = {On the local well-posedness of fractionally dissipated
              primitive equations with transport noise},
   JOURNAL = {J. Differential Equations},
  FJOURNAL = {Journal of Differential Equations},
    VOLUME = {450},
      YEAR = {2026},
     PAGES = {113729},
      ISSN = {0022-0396,1090-2732},
   MRCLASS = {76M35 (35Q35 35Q86 35R11 60H15 86A10)},
  MRNUMBER = {4951174},
       DOI = {10.1016/j.jde.2025.113729},
       URL = {https://doi.org/10.1016/j.jde.2025.113729},
}

@incollection{zhang2017backward,
  title={Backward stochastic differential equations},
  author={Zhang, Jianfeng},
  booktitle={Backward Stochastic Differential Equations: From Linear to Fully Nonlinear Theory},
  pages={79--99},
  year={2017},
  publisher={Springer}
}

@article {FG95,
    AUTHOR = {Flandoli, Franco and Gatarek, Dariusz},
     TITLE = {Martingale and stationary solutions for stochastic
              {N}avier-{S}tokes equations},
   JOURNAL = {Probab. Theory Related Fields},
  FJOURNAL = {Probability Theory and Related Fields},
    VOLUME = {102},
      YEAR = {1995},
    NUMBER = {3},
     PAGES = {367--391},
      ISSN = {0178-8051,1432-2064},
   MRCLASS = {60H15 (35Q30 35R60 76D05 76M35)},
  MRNUMBER = {1339739},
MRREVIEWER = {Marek\ Capi\'nski},
       DOI = {10.1007/BF01192467},
       URL = {https://doi-org.proxy.library.ucsb.edu/10.1007/BF01192467},
}

@article {CTV15,
    AUTHOR = {Constantin, Peter and Tarfulea, Andrei and Vicol, Vlad},
     TITLE = {Long time dynamics of forced critical {SQG}},
   JOURNAL = {Comm. Math. Phys.},
  FJOURNAL = {Communications in Mathematical Physics},
    VOLUME = {335},
      YEAR = {2015},
    NUMBER = {1},
     PAGES = {93--141},
      ISSN = {0010-3616,1432-0916},
   MRCLASS = {35Q86 (35B41)},
  MRNUMBER = {3314501},
MRREVIEWER = {Elisabetta\ Rocca},
       DOI = {10.1007/s00220-014-2129-3},
       URL = {https://doi-org.proxy.library.ucsb.edu/10.1007/s00220-014-2129-3},
}

@article {CWW15,
    AUTHOR = {Chae, Dongho and Wan, Renhui and Wu, Jiahong},
     TITLE = {Local well-posedness for the {H}all-{MHD} equations with
              fractional magnetic diffusion},
   JOURNAL = {J. Math. Fluid Mech.},
  FJOURNAL = {Journal of Mathematical Fluid Mechanics},
    VOLUME = {17},
      YEAR = {2015},
    NUMBER = {4},
     PAGES = {627--638},
      ISSN = {1422-6928,1422-6952},
   MRCLASS = {35Q35 (35B35 35B65 35R11 76W05)},
  MRNUMBER = {3412271},
       DOI = {10.1007/s00021-015-0222-9},
       URL = {https://doi.org/10.1007/s00021-015-0222-9},
}

@article {Dai21,
    AUTHOR = {Dai, Mimi},
     TITLE = {Local well-posedness for the {H}all-{MHD} system in optimal
              {S}obolev spaces},
   JOURNAL = {J. Differential Equations},
  FJOURNAL = {Journal of Differential Equations},
    VOLUME = {289},
      YEAR = {2021},
     PAGES = {159--181},
      ISSN = {0022-0396,1090-2732},
   MRCLASS = {76D03 (35Q35 76W05)},
  MRNUMBER = {4248458},
MRREVIEWER = {Abhik\ Kumar\ Sanyal},
       DOI = {10.1016/j.jde.2021.04.019},
       URL = {https://doi.org/10.1016/j.jde.2021.04.019},
}

@book{prevot2007concise,
  title={A concise course on stochastic partial differential equations},
  author={Pr{\'e}v{\^o}t, Claudia and R{\"o}ckner, Michael},
  year={2007},
  publisher={Springer}
}

@article {ANR11,
    AUTHOR = {Debussche, Arnaud and Glatt-Holtz, Nathan and Temam, Roger},
     TITLE = {Local martingale and pathwise solutions for an abstract fluids
              model},
   JOURNAL = {Phys. D},
  FJOURNAL = {Physica D. Nonlinear Phenomena},
    VOLUME = {240},
      YEAR = {2011},
    NUMBER = {14-15},
     PAGES = {1123--1144},
      ISSN = {0167-2789,1872-8022},
   MRCLASS = {60H15 (60H30)},
  MRNUMBER = {2812364},
MRREVIEWER = {Abdulrahman\ S.\ Al-Hussein},
       DOI = {10.1016/j.physd.2011.03.009},
       URL = {https://doi-org.proxy.library.ucsb.edu/10.1016/j.physd.2011.03.009},
}

@book {Doob94,
    AUTHOR = {Doob, J. L.},
     TITLE = {Measure theory},
    SERIES = {Graduate Texts in Mathematics},
    VOLUME = {143},
 PUBLISHER = {Springer-Verlag, New York},
      YEAR = {1994},
     PAGES = {xii+210},
      ISBN = {0-387-94055-3},
   MRCLASS = {28-01 (60A10)},
  MRNUMBER = {1253752},
MRREVIEWER = {S.\ D.\ Chatterji},
       DOI = {10.1007/978-1-4612-0877-8},
       URL = {https://doi.org/10.1007/978-1-4612-0877-8},
}

@article {Dai16,
    AUTHOR = {Dai, Mimi},
     TITLE = {Regularity criterion for the 3{D}
              {H}all-magneto-hydrodynamics},
   JOURNAL = {J. Differential Equations},
  FJOURNAL = {Journal of Differential Equations},
    VOLUME = {261},
      YEAR = {2016},
    NUMBER = {1},
     PAGES = {573--591},
      ISSN = {0022-0396,1090-2732},
   MRCLASS = {76D03 (35B65 35Q35 76W05)},
  MRNUMBER = {3487269},
MRREVIEWER = {Anthony\ Suen},
       DOI = {10.1016/j.jde.2016.03.019},
       URL = {https://doi.org/10.1016/j.jde.2016.03.019},
}

@article {WZ17,
    AUTHOR = {Wan, Renhui and Zhou, Yong},
     TITLE = {Low regularity well-posedness for the 3{D} generalized
              {H}all-{MHD} system},
   JOURNAL = {Acta Appl. Math.},
  FJOURNAL = {Acta Applicandae Mathematicae},
    VOLUME = {147},
      YEAR = {2017},
     PAGES = {95--111},
      ISSN = {0167-8019,1572-9036},
   MRCLASS = {35Q35 (35B30 35R11 76W05)},
  MRNUMBER = {3592797},
       DOI = {10.1007/s10440-016-0070-5},
       URL = {https://doi.org/10.1007/s10440-016-0070-5},
}

@book{bahouri2011fourier,
  title={Fourier analysis and nonlinear partial differential equations},
  author={Bahouri, Hajer},
  year={2011},
  publisher={Springer}
}

@book{grafakos2008classical,
  title={Classical fourier analysis},
  author={Grafakos, Loukas and others},
  volume={2},
  year={2008},
  publisher={Springer}
}

@article {AHHS25,
    AUTHOR = {Agresti, Antonio and Hieber, Matthias and Hussein, Amru and
              Saal, Martin},
     TITLE = {The stochastic primitive equations with nonisothermal
              turbulent pressure},
   JOURNAL = {Ann. Appl. Probab.},
  FJOURNAL = {The Annals of Applied Probability},
    VOLUME = {35},
      YEAR = {2025},
    NUMBER = {1},
     PAGES = {635--700},
      ISSN = {1050-5164,2168-8737},
   MRCLASS = {35Q86 (35Q35 35R60 60H15 76M35 76U60)},
  MRNUMBER = {4871718},
MRREVIEWER = {Adrian\ Muntean},
       DOI = {10.1214/24-aap2124},
       URL = {https://doi.org/10.1214/24-aap2124},
}

@article {BS21,
    AUTHOR = {Brzezniak, Zdzislaw and Slavik, Jakub},
     TITLE = {Well-posedness of the 3{D} stochastic primitive equations with
              multiplicative and transport noise},
   JOURNAL = {J. Differential Equations},
  FJOURNAL = {Journal of Differential Equations},
    VOLUME = {296},
      YEAR = {2021},
     PAGES = {617--676},
      ISSN = {0022-0396,1090-2732},
   MRCLASS = {35Q86 (60H15 76M35)},
  MRNUMBER = {4275634},
       DOI = {10.1016/j.jde.2021.05.049},
       URL = {https://doi.org/10.1016/j.jde.2021.05.049},
}

@article {CDL14,
    AUTHOR = {Chae, Dongho and Degond, Pierre and Liu, Jian-Guo},
     TITLE = {Well-posedness for {H}all-magnetohydrodynamics},
   JOURNAL = {Ann. Inst. H. Poincar\'e{} C Anal. Non Lin\'eaire},
  FJOURNAL = {Annales de l'Institut Henri Poincar\'e{} C. Analyse Non
              Lin\'eaire},
    VOLUME = {31},
      YEAR = {2014},
    NUMBER = {3},
     PAGES = {555--565},
      ISSN = {0294-1449,1873-1430},
   MRCLASS = {35Q35 (35B30 35B53 35L60 76W05)},
  MRNUMBER = {3208454},
MRREVIEWER = {Iuliana\ Oprea},
       DOI = {10.1016/j.anihpc.2013.04.006},
       URL = {https://doi-org.proxy.library.ucsb.edu/10.1016/j.anihpc.2013.04.006},
}

@article{kraichnan1968small,
  title={Small-scale structure of a scalar field convected by turbulence},
  author={Kraichnan, Robert H},
  journal={The Physics of Fluids},
  volume={11},
  number={5},
  pages={945--953},
  year={1968},
  publisher={American Institute of Physics}
}

@article{kraichnan1994anomalous,
  title={Anomalous scaling of a randomly advected passive scalar},
  author={Kraichnan, Robert H},
  journal={Physical review letters},
  volume={72},
  number={7},
  pages={1016},
  year={1994},
  publisher={APS}
}

@incollection{mikulevicius2001equations,
  title={On equations of stochastic fluid mechanics},
  author={Mikulevicius, Remigijus and Rozovskii, B},
  booktitle={Stochastics in Finite and Infinite Dimensions: In Honor of Gopinath Kallianpur},
  pages={285--302},
  year={2001},
  publisher={Springer}
}

@article {MR04,
    AUTHOR = {Mikulevicius, R. and Rozovskii, B. L.},
     TITLE = {Stochastic {N}avier-{S}tokes equations for turbulent flows},
   JOURNAL = {SIAM J. Math. Anal.},
  FJOURNAL = {SIAM Journal on Mathematical Analysis},
    VOLUME = {35},
      YEAR = {2004},
    NUMBER = {5},
     PAGES = {1250--1310},
      ISSN = {0036-1410,1095-7154},
   MRCLASS = {60H15 (35Q30 35R60 76D05 76F02 76M35)},
  MRNUMBER = {2050201},
MRREVIEWER = {Anna\ Karczewska},
       DOI = {10.1137/S0036141002409167},
       URL = {https://doi-org.proxy.library.ucsb.edu/10.1137/S0036141002409167},
}

@article{agresti2023primitive,
  title={The primitive equations with rough transport noise: Global well-posedness and regularity},
  author={Agresti, Antonio},
  journal={arXiv preprint arXiv:2310.01193},
  year={2023}
}

@article{hu2023local,
  title={Local martingale solutions and pathwise uniqueness for the three-dimensional stochastic inviscid primitive equations},
  author={Hu, Ruimeng and Lin, Quyuan},
  journal={Stochastics and Partial Differential Equations: Analysis and Computations},
  volume={11},
  number={4},
  pages={1470--1518},
  year={2023},
  publisher={Springer}
}

@article{hu2023pathwise,
  title={Pathwise solutions for stochastic hydrostatic Euler equations and hydrostatic Navier-Stokes equations under the local Rayleigh condition},
  author={Hu, Ruimeng and Lin, Quyuan},
  journal={arXiv preprint arXiv:2301.07810},
  year={2023}
}

@article{hu2025regularization,
  title={Regularization by noise for the inviscid primitive equations},
  author={Hu, Ruimeng and Lin, Quyuan and Liu, Rongchang},
  journal={Journal of Nonlinear Science},
  volume={35},
  number={4},
  pages={84},
  year={2025},
  publisher={Springer}
}

@article{abdo2023long,
  title={On the long-time dynamics and ergodicity of the stochastic Nernst-Planck-Navier-Stokes system},
  author={Abdo, Elie and Hu, Ruimeng and Lin, Quyuan},
  journal={arXiv preprint arXiv:2310.20484},
  year={2023}
}

@article{hu2025well,
  title={Well-posedness of the Electron MHD with random diffusion},
  author={Hu, Ruimeng and Peng, Qirui and Yang, Xu},
  journal={arXiv e-prints},
  pages={arXiv--2509},
  year={2025}
}

\appendix 

\section{functional analysis}\label{appendix:functional_analysis}
Let $X$ be a Banach space, we first recall here the definition of the function space $W^{\a,p}(0,T;X)$ :
\begin{equation}\label{def:Wap_space}
W^{\a,p}(0,T;X) := \Bigl \{ f \in L^p(0,T;X) : \int_0^T \int_0^T \frac{\|f(t)-f(s) \|^p_X}{|t-s|^{1+\a p}}dsdt < \infty   \Bigr \},
\end{equation}
with the norm given by: 
\begin{equation}\label{def:Wap_space_norm}
\|f\|^p_{W^{\a,p}(0,T;X)} := \int_0^T \|f(t)\|^p_X dt +\int_0^T \int_0^T \frac{\|f(t)-f(s) \|^p_X}{|t-s|^{1+\a p}}dsdt. 
\end{equation}
\begin{Theorem}\label{Thm:compactness_thm_1}
Let $X \subset Y \subset Z$ be Banach spaces such that $X$ and $Z$ are reflexive and the embedding $X \subset \subset Y$ is compact and $Y \hookrightarrow Z$ is continuous. Then for $p \in (1,\infty)$ and $\gm \in (0,1]$, we have the following compact embedding:
\[
L^p(0,T;X) \cap W^{\gm,p}(0,T;Z) \subset \subset L^p (0,T;Y).
\]
\end{Theorem}
\begin{proof}
See \cite[Theorem 2.1]{FG95}.
\end{proof}
\begin{Theorem}\label{Thm:compactness_thm_2}
Let $X \subset \subset Y$ be two Banach spaces such that the embedding is compact. For $\gm \in \left(0,1\right]$ and $p \in (1,\infty)$ satisfying 
\[
\gm p > 1,
\]
we have the following compact embedding:
\[
W^{\gm,p}(0,T;X) \subset \subset C (0,T;Y).
\]
\end{Theorem}
\begin{proof}
See \cite[Theorem 2.2]{FG95}. 
\end{proof}

\section{Tools from Stochastic Calculus}\label{appendix:stochastic_calculus}
\smallskip
\begin{Lemma}\label{lemma:Ito_formula_EMHD}
Let $p \geq 2$ and $B_n$ be the solution to the Galerkin system \eqref{eq:EMHD_Ito_Galerkin}, then it holds that \\
\begin{align}
d \twonorm{\lb^s B_n}^p &= -p \left< \chir^2 \Pp_n \left( \curl \left[ (\curl B_n) \times B_n \right ] + \mu \lb^\a B_n \right), \lb^{2s}B_n \right> \twonorm{\lb^s B_n}^{p-2} dt \notag \\
&+ \frac{p}{2} \sum_{k =1}^\infty \bigl ( \<\Pp_n \Tt^2_k B_n ,\lb^{2s} B_n\> + \twonorm{\lb^s \Pp_n \Tt_k B_n}^2 \bigr) \twonorm{\lb^s B_n}^{p-2} dt \notag \\
&+\haf p (p-2) \twonorm{\lb^s B_n}^{p-4} \sum_{k=1}^\infty \< \Pp_n \Tt_k B_n, \lb^{2s} B_n  \>^2 dt \notag \\
&+ p \twonorm{\lb^s B_n}^{p-2} \sum_{k =1 }^\infty \< \Pp_n \Tt_k B_n, \lb^{2s} B_n \> dW^k.  \label{}
\end{align}
\end{Lemma}

\begin{proof} By It\^{o} formula (see for example \cite[Theorem 4.32]{da2014stochastic}), we have that
\begin{align}
    dF(\lb^s B_n(t)) &= \sum_{k = 1}^\infty \< DF (\lb^s B_n(t)), \lb^s \Pp_n \Tt_k B_n dW^k \> + \<DF (\lb^s B_n(t)), \lb^s \phi(t) \> dt \notag \\
    &+ \haf \sum_{k=1}^\infty \tr \bigl ( D^2 F(\lb^s B_n(t)) (\lb^s \Pp_n \Tt_k B_n) (\lb^s \Pp_n \Tt_k B_n)^T \bigr )dt,  \label{lemma:Ito_formula_EMHD_1}
\end{align}
where 
\[
\phi(t) : = - \bigl( \chir^2 \Pp_n ( \curl ( (\curl B_n) \times B_n )) + \mu \lb^\a B_n \bigr) + \haf \sum_{k=1}^\infty \Pp_n \Tt^2_k B_n,
\]
and $F: H \to \R$ is given by
\[
F(u) := \twonorm {u}^p = \<u,u \>^{\frac{p}{2}}.
\]

We compute the derivatives of $F$ and begin with:
\[
DF(u) = p \twonorm{u}^{p-2} u, \ \ \  \big( D^2F(u) \big)_{ij} = p(p-2) \twonorm{u}^{\frac{p-4}{2}} u_i u_j + p \twonorm{u}^{\frac{p-2}{2}}I_{ij},
\]
where $I$ is the identity matrix. To see this, let $F(u) = \psi \circ \eta (u)$ with $\psi(z) := z^\frac{p}{2}$ and $\eta(u) = \<u,u\>$ and observe that 
\[
\psi' (z) = \haf p z^{\frac{p-2}{2}}, \ \ \ \psi''(z) = \frac{1}{4} p(p-2) z^{\frac{p-4}{2}},
\]
and 
\begin{align*}
D\eta(u) (h) &= \frac{\p}{\p t} \eta(u+th) = \lim_{t \to 0} \frac{\<u+th, u + th\>-\<u,u\>}{t} = \lim_{t \to \infty} ( 2\<u,h \> + t \<h,h \> ) = 2 \<u , h\>, \\
D^2\eta(u) (h,l) &= 2D (\<u,h\>) = 2 \lim_{t \to \infty} \frac{\<u+tl,h\> - \<u,h \>}{t} = 2 \<h,l\>.
\end{align*}
Then by chain rule we deduce that 
\begin{align*}
DF(u) &= \psi'(u) D\eta (u) = p (\eta(u))^{\frac{p-2}{2}} u = p \twonorm{u}^{p-2} u, \\
\big( D^2 F (u) \big)_{ij} &= \phi''(\eta(u)) (D\eta(u))_i (D\eta(u))_j + \phi'(\eta(u)) D^2 \eta(u) \\
&= \frac{p(p-2)}{4} \twonorm{u}^{p-4} u_i u_j + p \twonorm{u}^{\frac{p-2}{2}} I_{ij},
\end{align*}
from which the first term on the right hand side of \eqref{lemma:Ito_formula_EMHD_1} becomes
\begin{align*}
  \sum_{k = 1}^\infty \< DF (\lb^s B_n(t)), \lb^s \Pp_n \Tt_k B_n dW^k \> &= p \twonorm{\lb^s B_n}^{p-2} \sum_{k=1}^\infty \< \lb^s B_n, \lb^s \Pp_n \Tt_k B_n \> dW^k \\ 
  &= p \twonorm{\lb^s B_n}^{p-2} \sum_{k=1}^\infty \<\Pp_n \Tt_k B_n,\lb^{2s} B_n \> dW^k,
\end{align*}
where we have used the Plancherel's Theorem. As for the second term, we have 
\begin{align*}
    \< DF(\lb^s B_n(t)), \lb^s \phi (t) \> dt  &= -\bigl< DF(\lb^sB_n(t)),\lb^s \bigl( \chi^2_r \Pp_n (\curl ((\curl B_n)\times B_n)) +\mu \lb^\a B_n \bigr) \bigr> dt \\
    &+ \haf \sum_{k = 1}^\infty \<DF (\lb^s B_n(t)), \lb^s \Pp_n \Tt^2_k B_n) \>dt\\
    &=-p \twonorm{\lb^s B_n}^{p-2} \bigl< \chi^2_r \Pp_n (\curl ((\curl B_n)\times B_n) ) + \mu \lb^\a B_n , \lb^{2s} B_n \bigr \> dt \\
    &+\frac{p}{2} \twonorm{\lb^s B_n}^{p-2} \sum_{k=1}^\infty \left< \Pp_n \Tt^2_k B_n, \lb^{2s} B_n  \right>.
\end{align*}
Next, we investigate the third term,
\begin{align*}
&\haf \sum_{k=1}^\infty \tr \bigl( D^2 F(\lb^s B_n(t)) (\lb^s \Pp_n \Tt_k B_n) (\lb^s \Pp_n \Tt_k B_n)^T \bigr )dt  \\
=\ &\haf \sum_{k=1}^\infty \tr \bigl ( D^2 (F(\lb^s B_n(t)))_{ij} (\lb^s \Pp_n \Tt_k B_n)_i (\lb^s \Pp_n \Tt_k B_n)_j \bigr) dt \\
=\ &\frac{p(p-2)}{2} \twonorm{\lb^s B_n}^{p-4}\sum_{k=1}^\infty \<\lb^s B_n, \lb^s \Pp_n \Tt_k B_n \>^2dt+\frac{p}{2}\twonorm{\lb^s B_n}^{p-2}\sum_{k=1}^\infty  \twonorm{\lb^s \Pp_n \Tt_k B_n}^2dt \\
=\ &\frac{p(p-2)}{2} \twonorm{\lb^s B_n}^{p-4}\sum_{k=1}^\infty \< \Pp_n \Tt_k B_n,\lb^{2s} B_n  \>^2dt+\frac{p}{2}\twonorm{\lb^s B_n}^{p-2}\sum_{k=1}^\infty  \twonorm{\lb^s \Pp_n \Tt_k B_n}^2dt.
\end{align*}
Combining the expressions for all three terms of the above yields the result.
\end{proof}

\begin{Theorem}[Burkholder-Davis-Gundy Inequality]\label{Thm:Burkholder-Davis-Gundy Inequality} 
Let $p > 0$ and
\begin{equation}\label{Thm:Burkholder-Davis-Gundy Inequality_1}
M_t := \int_0^t \s_s dB_s, 
\end{equation}
where $B_s$ is a standard Brownian motion and $\s$ is such that 
\[
\E   \Bigl( \int_0^t \|\s_s\|^2_{L_2(\mathscr{U},X)} ds \Bigr)^\frac{p}{2}  < +\infty.
\]
Then there exist $c_p$ and $C_p$, depending the physical dimension and value $p > 0$, such that 
\begin{equation}\label{Thm:Burkholder-Davis-Gundy Inequality_2}
c_p \E \Bigl( \int_0^t \|\s_s\|^2_{L_2(\mathscr{U},X)} ds \Bigr)^\frac{p}{2}  < \E \left( |M^*_T|^p \right) < C_p \E \Bigl( \int_0^t \|\s_s\|^2_{L_2(\mathscr{U},X)} ds \Bigr)^\frac{p}{2} ,
\end{equation}
where
\[
M_T^* := \sup_{t \in [0,T]} |M_t|.
\]
In addition, for $p \geq 2$, $\a \in \left[0,\haf \right)$ and $\s_s \in L^p\bigl(\Om; L^p(0,T;L_2(\mathscr{U},X))\bigr)$, then
\begin{equation}\label{Thm:Burkholder-Davis-Gundy Inequality_3}
\E \| M_. \|^p_{W^{\a,p}(0,T;X)} \les_p \E \int_0^T \|\s_t \|^p_{L_2(\mathscr{U},X)} dt. 
\end{equation}
\end{Theorem}

\begin{proof}
See \cite[Theorem 2.4.1]{zhang2017backward} for \eqref{Thm:Burkholder-Davis-Gundy Inequality_2} and \cite[Lemma 2.1]{FG95} for \eqref{Thm:Burkholder-Davis-Gundy Inequality_3}.
\end{proof}

\begin{Corollary}\label{Coro:Martingale_inequality}
Let $\e, \k > 0$, if $M_t$ is a continuous local martingale as in \eqref{Thm:Burkholder-Davis-Gundy Inequality_1}, then we have that 
\[
\P \Bigl( \sup_{t\in [0,T]} |M_t| \geq \e \Bigr) \leq \frac{C}{\e} \E \Bigl( \Bigl( \int_0^T |\sigma_t|^2 dt \Bigr)^\haf \wedge \k \Bigr) + \P \Bigl( \Bigl ( \int_0^T |\sigma_t|^2 dt \Bigr)^\haf > \k \Bigr),
\]
for some positive constant $C$.
\end{Corollary}
\begin{proof}
The proof follows from \cite[Corollary~D.0.2]{prevot2007concise}. Let \[
\tau := \inf \Bigl \{ t \geq 0 : \Bigl( \int_0^t |\sigma_s|^2 ds \Bigr)^\haf  > \k \Bigr \} \wedge T.
\]
Then $\tau \leq T$ is an $\Ff_t -$stopping time. Hence by Theorem \ref{Thm:Burkholder-Davis-Gundy Inequality}, 
\begin{align*}
    \P \Bigl ( \sup_{t \in [0,T]} |M_t| \geq \e \Bigr ) &= \P \Bigl ( \sup_{t \in [0,T]} |M_t| \geq \e, \tau = T  \Bigr ) + \P \Bigl ( \sup_{t \in [0,T]} |M_t| \geq \e, \tau < T  \Bigr ) \\
    &\leq \frac{C}{\e} \E  \Bigl( \int_0^T |\sigma_t|^2 dt \Bigr)^\haf  + \P \Bigl( \sup_{t \in [0,T]} |M_t| \geq \e, \Bigl( \int_0^T |\sigma_t|^2 dt \Bigr)^\haf > \k \Bigr) \\
    &\leq \frac{C}{\e} \E \Bigl( \Bigl( \int_0^T |\sigma_t|^2 dt \Bigr)^\haf \wedge \k \Bigr) + \P \Bigl( \Bigl( \int_0^T |\sigma_t|^2 dt \Bigr)^\haf > \k \Bigr). 
\end{align*}
\end{proof}

\noindent If $H$ is a complete and separable metric space, and if $\Gg$ is a family of probability measures on $H$, then we say $\Gg$ is \textit{tight} if for any $\e > 0$, there exists a compact subset $K_\e \subseteq H$ such that 
\[
\mu(K_\e) \geq 1 - \e, \ \ \ \forall \e \in \Gg.
\]

\begin{Theorem}[Prokhorov] \label{Thm:Prokhorov}
Let $H$ be a complete and separable metric space. A family $\Gg$ of probability measures on $(H,\mathscr{B}(H))$ is relatively compact if and only if it is tight.
\end{Theorem}

\begin{proof}
See \cite[Theorem 2.3]{da2014stochastic}.
\end{proof}

\begin{Theorem}[Skorokhod's representation theorem] \label{Thm:Skorokhod}
Suppose that $H$ is a complete and separable metric space. Let $\{ \mu_n \}_{n \in \N}$ be a sequence of probability measure on $\mathscr{B} (H)$ converging weakly to a probability measure $\mu$. Then there exists a probability space $\left(\Omega,\mathscr{F},\P  \right)$, a sequence of random variables $\{X_n\}_{n \in \N}$ whose laws are $\{ \mu_n \}_{n \in \N}$ and a random variable $X$ whose law is $\mu$ such that 
\[
\lim_{n \to \infty} X_n = X, \ \ \ \P\text{-a.s.}
\]
\end{Theorem}
\begin{proof}
See \cite[Theorem 2.4]{da2014stochastic}.
\end{proof}

\section{Proof of Lemma \ref{Lemma:stochastic_convergence}}\label{appendix:proof_lemma_stochastic_convergence}
Here we provide the proof of Lemma \ref{Lemma:stochastic_convergence}, which follows closely to \cite{ANR11}. In the following, we will use $|\cdot|$ for $|\cdot|_X$ norm and $\|\cdot \|$ for $|\cdot|_{L_2(\mathscr{U},X)}$ norm. First, we denote the sequence of the stochastic integral and their limit by
\[
\ell^n := \int_0^t G^n dW^n = \sumk \int_0^t G^n_k dW^n_k, \ \ \ \ell := \int_0^t G dW = \sumk \int_0^t G_k dW_k.
\]
In addition, we denote their Fourier truncations by 
\[
\ell^n_N := \sum_{k = 1}^N \int_0^t G^n_k dW^n_k, \ \ \ \ell_N := \sum_{k = 1}^N \int_0^t G_k dW_k.
\]
Next we consider 
\begin{align}
\sup_{\tau \in [0,t]} \left| \ell^n - \ell \right| (\tau) &\leq  \sup_{\tau \in [0,t]} \left| \ell^n - \ell^n_N \right| (\tau) + \sup_{\tau \in [0,t]} \left| \ell^n_N - \ell_N \right| (\tau) + \sup_{\tau \in [0,t]} \left| \ell_N - \ell \right| (\tau). \label{proof_lemma:stochastic_convergence_1}
\end{align}
It suffices to show that the above terms converge in probability, i.e. that there exists $N_1 > 0$ such that

\begin{enumerate}[label=(\roman*)]
\item For all $n > N > N_1$, we have 
\[
\P \Bigl( \sup_{\tau \in [0,t]} \left| \ell^n - \ell^n_N \right| (\tau) > \e \Bigr) < \frac{\d}{3}.  \label{proof_lemma:stochastic_convergence_2}
\]
\item For all $n > N_1$ and all $k \in \N$, there holds 
\[
\P \Bigl( \sup_{\tau \in [0,t]} \left| \ell^n_k - \ell_k \right| (\tau) > \e \Bigr) < \frac{\d}{3}.  \label{proof_lemma:stochastic_convergence_3}
\]
\item For $N > N_1$, 
\[
\P \Bigl( \sup_{\tau \in [0,t]} \left| \ell - \ell_N \right| (\tau) > \e \Bigr) < \frac{\d}{3}.  \label{proof_lemma:stochastic_convergence_4}
\]
\end{enumerate}
for any $\e, \d >0$. Now fixing $\e$ and $\d$ positive, with Corollary \ref{Coro:Martingale_inequality} we have that for any $\k > 0$: 
\begin{align*}
    \P \Bigl( \sup_{\tau \in [0,t]} \left| \ell^n - \ell^n_N \right| (\tau) > \e \Bigr) &\leq \frac{C \k}{\e} + \P \Bigl( \sum_{k \geq N} \Bigl ( \int_0^T \|G^n_k\|^2 dt  \Bigr)^\haf > \k \Bigr),
\end{align*}
for some positive constant $C$. In particular, letting $\k = \frac{\d \e}{6C}$ we obtain from the above that 
\begin{align*}
\P \Bigl( \sup_{\tau \in [0,t]} \bigl| \ell^n - \ell^n_N \bigr| (\tau) > \e \Bigr) &\leq \frac{\d}{6} + \P \Bigl( \sum_{k \geq N} \Bigl ( \int_0^T \|G^n_k - G_k\|^2 dt  \Bigr)^\haf > \frac{\k}{2} \Bigr) \\
& \ + \P \Bigl( \sum_{k \geq N} \Bigl ( \int_0^T \|G_k\|^2 dt  \Bigr)^\haf > \frac{\k}{2} \Bigr) \\
&\ \ = \frac{\d}{6} + \P \Bigl( \sum_{k \geq N} \Bigl ( \int_0^T \|G^n_k - G_k\|^2 dt  \Bigr)^\haf > \frac{\d \e}{12 C} \Bigr) \\
&\ \ \ + \P \Bigl( \sum_{k \geq N} \Bigl ( \int_0^T \|G_k\|^2 dt  \Bigr)^\haf > \frac{\d \e}{12 C} \Bigr).
\end{align*}
Since $G^n \to G$ in probability in $C([0,T];L_2(\mathscr{U},X))$, there exists $N_1 > 0$ such that for all $n > N > N_1$,
\[
\P \Bigl( \sup_{\tau \in [0,t]} \left| \ell^n - \ell^n_N \right| (\tau) > \e \Bigr) \leq \frac{\d}{3},
\]
from which we conclude \ref{proof_lemma:stochastic_convergence_2}. The argument for \ref{proof_lemma:stochastic_convergence_4} is followed by a similar manner. \\\\
\noindent In order to show \ref{proof_lemma:stochastic_convergence_3}, we will need to mollify $G^n_k$ and $G_k$ in time. For fixed $\l > 0$, we let $\tphi_\l$ be the standard mollifier on $\R$, i.e., $\tphi (0) = 1, \; \tphi \in C_0^\infty (\R) \; \text{and} \;\int_\R \tphi (t) dt = 1,$ such that $\tphi_\l (t) := \frac{1}{\l} \tphi \left(\frac{t}{\l} \right), \; \l >0. $
Furthermore, we define
\[
\tJ_\l (G):= \frac{1}{\l} \int_0^t \tphi_\l(t-s) G(s) ds, \ \ \ G \in C ([0,T];X), \ \ \ \l >0.   
\]
Now for any fixed $k \in \N$, using integration by parts we obtain
\begin{align}
    \Bigl | \ell^n_k - \ell_k  \Bigr| &= \sup_{t \in [0,T]} \Bigl | \int_0^t G^n_k d W^n_k - \int_0^t G_k dW_k \Bigr | \notag \\
    & \leq \sup_{t \in [0,T]} \Bigl | \int_0^t \left( G^n_k - \tJ_\l (G^n_k) \right) dW^n_k  \Bigr| + \sup_{t \in [0,T]} \Bigl | \int_0^t \Bigl( \tJ_\l (G_k) - G_k \Bigr) dW_k  \Bigr| \notag \\ 
    &+ \sup_{t \in [0,T]} \Bigl | \int_0^t \tJ_\l (G^n_k) dW^n_k - \int_0^t \tJ_\l (G_k) dW_k  \Bigr | \notag \\
    & \leq \sup_{t \in [0,T]} \Bigl | \int_0^t \left( G^n_k - \tJ_\l (G^n_k) \right) dW^n_k \Bigr| + \sup_{t \in [0,T]} \Bigl | \int_0^t \left( \tJ_\l (G_k) -G_k \right) dW_k \Bigr| \notag \\
    & + \sup_{t \in [0,T]} \Bigl | \tJ_\l (G^n_k) W^n_k - \tJ_\l (G_k) W_k \Bigr|  \notag \\ 
    &+ \frac{1}{\l} \sup_{t \in [0,T]} \Bigl | \int_0^t \left( \tJ_\l (G_k) W_k - \tJ_\l (G^n_k) W^n_k\right) ds \Bigr | + \frac{1}{\l} \sup_{t \in [0,T]} \Bigl | \int_0^t \left( G_k W_k - G^n_k W^n_k\right) ds \Bigr |. \label{proof_lemma:stochastic_convergence_5}
\end{align}
Fixing $\e, \d >0$, the first term on the right hand side of \eqref{proof_lemma:stochastic_convergence_5} is estimated by 
\begin{align}
 & \ \  \P \Bigl( \sup_{t \in [0,T]} \Bigl| \int_0^t \bigl(G^n_k - \tJ_\l (G^n_k) \bigr) dW^n_k \Bigr| > \e \Bigr) \notag \\
 & \leq \frac{\d}{20} + \P \Bigl( \Bigl( \int_0^T \left \| G^n_k - G_k \right \|^2 dt \Bigr)^\haf > \frac{\d \e}{60C} \Bigr) 
 + \P \Bigl( \Bigl( \int_0^T \bigl \| G_k - \tJ_\l (G_k) \bigr \|^2 dt \Bigr)^\haf > \frac{\d \e}{60C} \Bigr) \notag \\ 
 &+ \P \Bigl( \Bigl( \int_0^T \bigl \| \tJ_\l(G_k) - \tJ_\l(G^n_k) \bigr \|^2 dt \Bigr)^\haf > \frac{\d \e}{60C} \Bigr), \notag 
\end{align}
where we used again Corollary \ref{Coro:Martingale_inequality}. From the convergence $G^n_k \to G_k$ in $C([0,T];L_2(\mathscr{U},X))$ and the property for standard mollifier, by choosing $\l > 0$ small enough, we can conclude that there exists $N_1 \in \N$ such that for $n > N_1$,
\begin{equation}\label{proof_lemma:stochastic_convergence_6}
\P \Bigl( \sup_{t \in [0,T]} \Bigl| \int_0^t \bigl(G^n_k - \tJ_\l (G^n_k) \bigr) dW^n_k \Bigr| > \e \Bigr) < \frac{\d}{15}.
\end{equation}
Similarly one has that for $n >N_1$, 
\begin{equation}\label{proof_lemma:stochastic_convergence_7}
\P \Bigl( \sup_{t \in [0,T]} \Bigl| \int_0^t \bigl( \tJ_\l( G_k ) - G_k \bigr) dW^n_k \Bigr| > \e \Bigr) < \frac{\d}{15}.
\end{equation}
Next we consider 
\begin{align*}
    &\P \Bigl( \sup_{t \in [0,T]} \bigl| \tJ_\l (G^n_k) W^n_k - \tJ_\l (G_k) W_k \bigr| > \e \Bigr) \\
    \leq\ &\P \Bigl( \sup_{t \in [0,T]} \bigl| \tJ_\l (G^n_k) W^n_k - \tJ_\l (G_k) W^n_k \bigr| > \frac{\e}{2} \Bigr) 
+ \P \Bigl( \sup_{t \in [0,T]} \bigl| \tJ_\l (G_k) W^n_k - \tJ_\l (G_k) W_k \bigr| > \frac{\e}{2} \Bigr) \\
    \leq \ & \P \Bigl( \sup_{t \in [0,T]}  \bigl |W^n_k  \bigr|_{\mathscr{U}} \sup_{t \in [0,T]}  \bigl \| \tJ_\l (G^n_k) - \tJ_\l (G_k)\bigr \| > \frac{\e}{2} \Bigr) + \P \Bigl( \sup_{t \in [0,T]}  \bigl |W^n_k - W_k \bigr|_{\mathscr{U}} \sup_{t \in [0,T]}  \bigl \| \tJ_\l (G_k)\bigr \| > \frac{\e}{2} \Bigr).
\end{align*}
Again, from the convergence $W^n \to W$, $G^n_k \to G_k$ and the property of the standard mollification, we conclude for $n > N_1$ it holds that 
\begin{equation}\label{proof_lemma:stochastic_convergence_8}
\P \Bigl( \sup_{t \in [0,T]} \bigl| \tJ_\l (G^n_k) W^n_k - \tJ_\l (G_k) W_k \bigr| > \e \Bigr) < \frac{\d}{15}.
\end{equation}
The last two terms in \eqref{proof_lemma:stochastic_convergence_5} are estimated by a similar manner:
\begin{align*}
   &\P \Bigl(  \frac{1}{\l} \sup_{t \in [0,T]} \Bigl | \int_0^t \bigl( \tJ_\l (G_k) W_k - \tJ_\l (G^n_k) W^n_k\bigr) ds \Bigr | > \e \Bigr) \\
\leq \ & \P \Bigl(  \frac{1}{\l} \sup_{t \in [0,T]} \Bigl | \int_0^t \bigl( \tJ_\l (G_k) W^n_k - \tJ_\l (G^n_k) W^n_k\bigr) ds \Bigr | > \frac{\e}{2} \Bigr) +\P \Bigl(  \frac{1}{\l} \sup_{t \in [0,T]} \Bigl | \int_0^t \bigl( \tJ_\l (G_k) W^n_k - \tJ_\l (G^n_k) W^n_k\bigr) ds \Bigr | > \frac{\e}{2} \Bigr) \\
\leq \ & \P \Bigl( \sup_{t \in [0,T]}   |W^n_k |_{\mathscr{U}} \sup_{t \in [0,T]}  \bigl \| \tJ_\l (G^n_k) - \tJ_\l (G_k)\bigr \| > \frac{\e \l}{2T} \Bigr) + \P \Bigl( \sup_{t \in [0,T]}  |W^n_k - W_k |_{\mathscr{U}} \sup_{t \in [0,T]}  \bigl \| \tJ_\l (G_k)\bigr \| > \frac{\e \l}{2 T} \Bigr).
\end{align*} 
Therefore the conclusion holds as in $\eqref{proof_lemma:stochastic_convergence_8}$ by choosing $\l >0$ small enough: 
\begin{equation}\label{proof_lemma:stochastic_convergence_9}
\P \Bigl(  \frac{1}{\l} \sup_{t \in [0,T]} \Bigl | \int_0^t \left( \tJ_\l (G_k) W_k - \tJ_\l (G^n_k) W^n_k\right) ds \Bigr | > \e \Bigr) < \frac{\d}{15}, \ \ \text{for} \ \ n > N_1.
\end{equation}
and similarly
\begin{equation}\label{proof_lemma:stochastic_convergence_10}
\P \Bigl(  \frac{1}{\l} \sup_{t \in [0,T]} \Bigl | \int_0^t \left( G_k W_k - G^n_k W^n_k\right) ds \Bigr | > \e \Bigr) < \frac{\d}{15}, \ \ \text{for} \ \ n > N_1.
\end{equation}
Combining \eqref{proof_lemma:stochastic_convergence_5}-\eqref{proof_lemma:stochastic_convergence_10}, we finish the proof of \ref{proof_lemma:stochastic_convergence_3} and the lemma follows.    \qed

\section{Commutator estimates}\label{appendix:commutator_estimates}

\begin{Lemma}[\cite{CTV15}]\label{lemma:commutator_est_0}
Let $s > 0$ and $f,g \in H^s \cap W^{1,\infty}$. We have that 
\[
\twonorm{\lb^s (fg)} \les \Linfnorm{f} \Hsnorm{g} + \Linfnorm{g} \Hsnorm{f},
\]
and 
\[
\twonorm{\lb^s(fg)-f \lb^s g} \les \Linfnorm{\nabla f} \twonorm{\lb^{s-1} g}+\twonorm{\lb^{s}f}\Linfnorm{g}.
\]
\end{Lemma}

\begin{Lemma}\label{lemma:commutator_est_1}
Let $s > \frac{5}{2}$ and $\gm > -s$. Then for any smooth divergence- free vector field $b$ and $u$ on $\T^3$, we have that
\[
\twonorm{\lb^{\gm} \left [\lb^s, b \cdot \nabla  \right ] u } \les \Hsnorm{b} \|u \|_{H^{s+\gm}(\T^3)} + \|b \|_{H^{s+\gm}(\T^3)} \Hsnorm{u}.
\]
\end{Lemma}
\begin{proof}
See \cite[Lemma B.1]{HLL26}. 
\end{proof}

\begin{Lemma}\label{lemma:commutator_est_2}
Let $s > \frac{5}{2}$ and $b$ be a smooth divergence-free vector field on $\T^3$ with zero mean. Then it holds that 
\[
\twonorm{\left[ \left[\lb^s, b\cdot \nabla \right ], b \cdot \nabla \right] u } \les \| b\|^2_{H^{s+1}(\T^3)} \Hsnorm{u}.
\]
\end{Lemma}
\begin{proof}
See \cite[Lemma B.2]{HLL26}.
\end{proof}

\begin{Lemma}\label{lemma:commutator_est_3}
For tempered distributions $u$ and $v$, it follows that 
\[
\rnorm{\left [ \Dq, \Ss_{l-2} u \cdot \nabla \right] \Dl v } \les \Linfnorm{\nabla \Ss_{l-2} u} \rnorm{\Dl v},
\]
for any $r \in \left (1, \infty \right)$.
\end{Lemma}
\begin{proof}
See \cite[Lemma 2.5]{Dai16}.
\end{proof}

\end{document}